\documentclass{article}
\usepackage[utf8]{inputenc}

\title{Topoi with enough points and topological groupoids}
\author{J. L. Wrigley\footnote{School of Mathematical Sciences, Queen Mary University of London, Mile End Road, London E1 4NS, UK, email: \texttt{j.wrigley@qmul.ac.uk}}}

\usepackage[only,llbracket,rrbracket,shortleftarrow,shortrightarrow,mapsfrom]{stmaryrd}

\usepackage[dvipsnames]{xcolor}
\definecolor{Forest}{rgb}{0.13, 0.55, 0.13}
\definecolor{Darkcyan}{rgb}{0.0, 0.55, 0.55}

\usepackage[pdfencoding=auto,pdfusetitle]{hyperref}
\hypersetup{colorlinks=true, linkcolor={black}, citecolor={black}, urlcolor={black}}

\usepackage{amssymb}
\usepackage{amsmath}
\usepackage{amsthm}
\usepackage{tikz-cd}
\usepackage{fancyhdr}
\usepackage{mathtools}
\usepackage{enumitem}
\usepackage{extpfeil}
\usepackage{bussproofs}
\usepackage{mathtools}
\usepackage{scalerel}
\usepackage{booktabs}

\usepackage[noabbrev,capitalise]{cleveref}
\crefformat{equation}{(#2#1#3)}
\crefformat{enumi}{#2#1#3}
\crefformat{enumii}{#2#1#3}

\RequirePackage{tikz-cd}
\RequirePackage{amssymb}
\usetikzlibrary{calc}
\usetikzlibrary{decorations.pathmorphing}

\tikzset{curve/.style={settings={#1},to path={(\tikztostart)
			.. controls ($(\tikztostart)!\pv{pos}!(\tikztotarget)!\pv{height}!270:(\tikztotarget)$)
			and ($(\tikztostart)!1-\pv{pos}!(\tikztotarget)!\pv{height}!270:(\tikztotarget)$)
			.. (\tikztotarget)\tikztonodes}},
	settings/.code={\tikzset{quiver/.cd,#1}
		\def\pv##1{\pgfkeysvalueof{/tikz/quiver/##1}}},
	quiver/.cd,pos/.initial=0.35,height/.initial=0}

\tikzset{tail reversed/.code={\pgfsetarrowsstart{tikzcd to}}}
\tikzset{2tail/.code={\pgfsetarrowsstart{Implies[reversed]}}}
\tikzset{2tail reversed/.code={\pgfsetarrowsstart{Implies}}}
\tikzset{no body/.style={/tikz/dash pattern=on 0 off 1mm}}

\tikzset{
	labl/.style={anchor=south, rotate=90, inner sep=.5mm}
}

\setlist{listparindent = \parindent, parsep=0pt,}
\setenumerate[1]{label = (\roman*), ref = (\roman*)}

\theoremstyle{plain}
\newtheorem{thm}{Theorem}
\newtheorem{lem}{Lemma}[section]
\newtheorem{coro}[lem]{Corollary}
\newtheorem{prop}[lem]{Proposition}

\theoremstyle{definition}
\newtheorem{df}[lem]{Definition}
\newtheorem{dfs}[lem]{Definitions}
\newtheorem{rem}[lem]{Remark}

\newtheorem{ex}[lem]{Example}
\newtheorem{exs}[lem]{Examples}
\newtheorem{fact}[lem]{Fact}
\newtheorem{nota}[lem]{Notation}

\newtheoremstyle{repeat}{}{}{\itshape}{}{\bfseries}{.}{.5em}{#3, bis}
\theoremstyle{repeat}
\newtheorem*{repeated-theorem}{Repeat}

\renewcommand{\phi}{\varphi}

\newcommand{\op}{^{\rm op}}
\newcommand{\id}{{\rm{id}}}

\newcommand{\0}{{\bf 0}}

\newcommand{\pr}{{\text{pr}}}
\newcommand{\lrset}[2]{\left\{\,{#1}\,\middle\vert\,{#2}\,\right\}}
\newcommand{\lrangle}[1]{\left \langle {#1} \right \rangle}

\newcommand{\Q}{\mathbb{Q}}
\newcommand{\Z}{\mathbb{Z}}

\newcommand{\cat}{\mathcal{C}}
\newcommand{\dcat}{\mathcal{D}}

\newcommand{\theory}{\mathbb{T}}

\newcommand{\topos}{\mathcal{E}}
\newcommand{\ftopos}{\mathcal{F}}

\newcommand{\CAT}{{\bf CAT}}

\newcommand{\sets}{{\bf Sets}}

\newcommand{\Top}{{\bf Top}}

\newcommand{\Topos}{{\bf Topos}}
\newcommand{\wep}{\textit{w.e.p.}}
\newcommand{\WEP}{\wep}
\newcommand{\iso}{\mathrm{iso}}

\newcommand{\Sh}{{\bf Sh}}
\newcommand{\opens}{\mathcal{O}}

\newcommand{\Sub}{{\rm{Sub}}}

\newcommand{\Pt}{{\rm Pt}}

\newcommand{\TopGrpd}{{\bf TopGrpd}}
\newcommand{\LogGrpd}{{\bf LogGrpd}}
\newcommand{\ECTGrpd}{{\bf ECTGrpd}}

\newcommand{\B}{{\bf B}}

\newcommand{\X}{\mathbb{X}}
\newcommand{\Y}{\mathbb{Y}}
\newcommand{\U}{\mathbb{U}}

\newcommand{\W}{\mathbb{W}}
\newcommand{\V}{\mathbb{V}}

\newcommand{\Xtt}{\mathbb{X}_{\tau_0}^{\tau_1}}

\newcommand{\Aut}{{\rm Aut}}

\newcommand{\form}[2]{\{\,\vec{{#2}} : {#1}\,\}}

\newcommand{\class}[1]{\llbracket\, {#1} \,\rrbracket}
\newcommand{\classv}[2]{\llbracket\, \vec{{#2}} : {{#1}} \,\rrbracket}
\newcommand{\lrclass}[1]{\left\llbracket\, {#1} \,\right\rrbracket}

\newcommand{\Tmodels}[1]{\mathbb{T}\text{-}{\bf Mod}({{#1}})}

\newcommand{\weqv}{\mathfrak{W}}

\newcommand{\Index}{\mathfrak{K}}

\newcommand{\tp}{{\rm tp}}

\newcommand{\paronto}{%
	\rightharpoondown\mathrel{\mspace{-15mu}}\rightharpoondown
}

\usepackage{geometry}
\begin{document}
	
	\maketitle
	
	\begin{abstract}
		We establish a bi-equivalence between the bi-category of topoi with enough points and a localisation of a bi-subcategory of topological groupoids.  
	\end{abstract}
	
	\renewcommand{\thefootnote}{\fnsymbol{footnote}} 
	\footnotetext{\emph{2020 Mathematics Subject Classification:} 18B25 (Primary), 03G30, 22A22 (Secondary).}
	\footnotetext{\emph{Key words and phrases:} Grothendieck topos, topos with enough points, topological groupoid.}     
	\renewcommand{\thefootnote}{\arabic{footnote}} 
	
	
	
	\section*{Introduction}
	
	\paragraph{Toposes as generalised spaces.}  The notion of a Grothendieck topos (hereafter, just topos), central to many developments in algebraic geometry, is often described as a `generalised topological space', where points are allowed to have isomorphisms.  More precisely, every topos \emph{with enough points} is \emph{represented}, non-uniquely, by a topological groupoid (see \cite{BM}).  If the topos lacks enough points, we must instead use localic (i.e.\ point-free) groupoids (see \cite{JT}).  In practice, most topoi of interest have enough points.  This raises the question: to what extent are topological groupoids a substitute to the study of topoi?
	
	In the localic case, this has been satisfactorily answered: Moerdijk demonstrated that the category of topoi is equivalent to a localisation of a subcategory of localic groupoids by a right calculus of fractions \cite[Theorem 7.7]{cont1}.  
	
	\paragraph{Our contribution.}

	There are many reasons to desire a topological (i.e.\ point-set) analogue of Moerdijk's result for topoi with enough points (see, for instance, \cref{rem:geo-is-hom} or \cref{rem:advert-for-AZ-extension}).  But the passage from localic to topological groupoids dismantles Moerdijk's techniques.  We cannot obtain an equivalence between the full category of topoi with enough points and a localisation of a subcategory of topological groupoids by a \emph{right} calculus of fractions (as shown in \cref{prop:no_right_calc_for_top}).
	
	Instead, this paper demonstrates that a (bi-)equivalence is possible with a localisation by a \emph{left} (bi-)calculus of fractions.  Consequently, the (bi-)equivalence we establish has a markedly different flavour to the localic version.
	\begin{thm}\label{thm:main-thm-intro}
		There is a bi-equivalence
		\[
		\Topos_\wep^\iso \simeq [\weqv^{-1}]\ECTGrpd
		\]
		where:
		\begin{enumerate}
			\item $\Topos_\wep^\iso$ denotes the bi-category of topoi with enough points, geometric morphisms and natural isomorphisms,
			\item and $\ECTGrpd$, the \emph{\'etale complete topological groupoids}, is a full bi-subcategory of the bi-category of topological groupoids, continuous functors, and continuous natural transformations, and $[\weqv^{-1}]\ECTGrpd$ denotes the localisation by a left bi-calculus of fractions $\weqv$ on $\ECTGrpd$.
		\end{enumerate}
	\end{thm}

	\paragraph{Overview.}  We proceed as follows.
	\begin{enumerate}[label = (\arabic*)]
		\item First, some preliminaries are recalled in \cref{sec:prelims}.  We recall the construction of the bi-category of topological groupoids $\TopGrpd$ and the bi-functor $\Sh \colon \TopGrpd \to \Topos_\WEP^\iso$ that sends a topological groupoid to its \emph{topos of sheaves}, as well as the interplay between this method of presenting topoi and the theory of classifying topoi, as found in \cite{myselfgrpds}, which will be exploited throughout the paper.  Also in \cref{sec:prelims}, we recall the construction of a bi-category of fractions from \cite{pronk}.
		
		\item Before embarking on the proof of \cref{thm:main-thm-intro}, we demonstrate that it is impossible to obtain a bi-equivalence for $\Topos_\WEP^\iso$ with a \emph{right} bi-calculus of fractions on a bi-subcategory of topological groupoids.  Thus, the topological case contrasts with the localic case considered in \cite[\S 7]{cont1}.
		
		\item In \cref{sec:loggrpd}, we introduce the study of `logical groupoids', a special class of topological groupoids.  Although we will eventually need to restrict further to the étale complete groupoids, the results of \cref{sec:weqv} hold in the weaker setting of logical groupoids.
		
		\item \cref{sec:weqv} contains the bulk of the proof of \cref{thm:main-thm-intro}.  We identify which inclusions of subgroupoids of logical groupoids induce equivalences of their sheaf topoi (\cref{classification-of-weqv}), and demonstrate that every geometric morphism is induced by a continuous of logical groupoids.  We also show that the essential image factorisation of a continuous functor yields the surjection-inclusion factorisation of the induced geometric morphism.
		 
		\item In \cref{sec:ectgrpd}, we further restrict our attention to \emph{étale complete groupoids}, which are the objects of $\ECTGrpd$.  In particular, we will describe the \emph{étale completion} of a (sober) logical groupoid.
		
		\item Finally, the various ingredients introduced in the preceding sections are combined in \cref{sec:bieqv} to deduce the bi-equivalence \cref{thm:main-thm-intro}.
	\end{enumerate}

	\section{Preliminaries}\label{sec:prelims}
	
	We first recall some preliminaries.  We will assume familiarity with the bi-category of $\Topos$, the bi-category whose objects are topoi, whose morphisms are geometric morphisms, and whose 2-cells are natural transformations between the inverse image functors of these geometric morphisms.  We use $\Topos^\iso$ to denote the bi-subcategory of $\Topos$ with only the invertible 2-cells.
	
	\begin{df}
		A topos $\topos$ is said to \emph{have enough points} if the class of geometric morphisms $\sets \to \topos$ is jointly surjective, i.e.\ the inverse image functors are jointly conservative.  Recall from \cite[Corollary 7.17]{topos} that if $\topos$ has enough points, there is a mere set of jointly surjective geometric morphisms $\lrset{f_i \colon \sets \to \topos}{i \in I}$.  Let $\Topos_\wep$ (respectively, $\Topos_\wep^\iso$) denote the full bi-subcategory of $\Topos$ (resp., $\Topos^\iso$) on topoi with enough points.
	\end{df}
	
	We now introduce the further aspects of topos theory needed to understand this paper.
	\begin{enumerate}[label = (\arabic*)]
		\item We begin with the notion of a \emph{topos of sheaves} on a topological groupoid, and how this construction can be made bi-functorial.
		\item Next, we review the results of \cite{myselfgrpds} concerning which open topological groupoids can represent the \emph{classifying topos} of a theory.
		\item Finally, we recall from \cite{pronk} the construction of the localisation of a bi-category by a bi-calculus of fractions.
	\end{enumerate}
	
	\subsection{Sheaves on a topological groupoid}
	
	The theory of topological groupoids is simply the theory of groupoids internal to $\Top$.  We spell out what this means:
	\begin{dfs}
		\begin{enumerate}
			\item A topological groupoid $\X = (X_1 \rightrightarrows X_0)$ is a (small) groupoid where the set of objects $X_0$ and the set of arrows $X_1$ have been endowed with topologies such that the maps:
			\begin{enumerate}
				\item $s \colon X_1 \to X_0$ and $t \colon X_1 \to X_0$ that send an arrow to, respectively, its source and target,
				\item $e \colon X_0 \to X_1$ that sends an object to the identity arrow on itself,
				\item $i \colon X_1 \to X_1$ that sends an arrow to its inverse,
				\item and the map $m \colon X_1 \times_{X_0} X_1 \to X_1$ that send a composable pair of arrows to their composite
			\end{enumerate}
			are all continuous.  If $s$ (or, equivalently, $t$) is open as well as continuous, we say that $\X$ is an \emph{open} topological groupoid.
			
			\item A \emph{continuous functor} of topological groupoids is a functor $\Phi \colon \X \to \Y$ such that the action on objects $\Phi_0 \colon X_0 \to Y_0$ and the action on arrows $\Phi_1 \colon X_1 \to Y_1$ are both continuous maps.
			
			\item A \emph{continuous transformation} $a \colon \Phi \Rightarrow \Psi$ between continuous functors is a natural transformation of the underlying functors such that the map $a \colon X_0 \to Y_1$ that sends an object to its component $a_x \colon \Phi_0(x) \to  \Psi_0(x)$ is continuous.
		\end{enumerate}
		Together this data defines a bi-category that we denote by $\TopGrpd$.
	\end{dfs}
	\begin{exs}\label{exs:example-grpds}
		\begin{enumerate}
			\item Each topological space $X$ yields a topological groupoid, with only identity arrows on the objects $X$, i.e.\ a categorically discrete groupoid.
			\item Each groupoid $\X$ (in the usual sense) can be viewed as a topologically discrete topological groupoid, where the topologies on both objects and arrows are the discrete topologies.
			\item Each topological group $G$ yields a topological groupoid with one object in the obvious way.
			\item A continuous action by a topological group $G$ on a space $X$ yields a topological groupoid.  The objects are $X$, and the arrows $G \times X$.  The source of a pair $(g,x) \in G \times X$ is $x$ while its target is the result of the action $g \cdot x$.
		\end{enumerate}
	\end{exs}
	\begin{dfs}
	Each topological groupoid comes naturally equipped with a notion of a \emph{topos of sheaves}.  Let $\X = (X_1 \rightrightarrows X_0)$ be a topological groupoid.
	\begin{enumerate}
		\item A \emph{sheaf} on $\X$ consists of a local homeomorphism $q \colon Y \to X_0$ and  an $X_1$-action $\beta \colon X_1 \times_{X_0} Y \to Y$, i.e.\ a continuous map such that, for all $y \in Y$ and compatible $g,h \in X_1$,
		\[
		\beta(h,\beta(g,y)) = \beta(h \circ g,y), \quad q(\beta(g,y)) = t(g), \quad \beta(e(q(y)),y) = y.
		\]
		\item A morphism of sheaves $f \colon (Y,q,\beta) \to (Y',q',\beta')$ consists of a continuous map $f \colon Y \to Y'$ such that $q' \circ f = q$ and $f$ is equivariant in the sense that $\beta'(g,f(y)) = f(\beta(g,y))$.
	\end{enumerate}
	Together, the sheaves on $\X$ and their morphisms define a topos, which we denote by $\Sh(\X)$.
	\end{dfs}
	\begin{nota}\label{notation-for-sheaves}
		\begin{enumerate}
			\item When we wish to keep track of the topologies $\tau_0$ and $\tau_1$ on $X_0$ and $X_1$ making $\X$ a topological groupoid, we will write $\Xtt$ for the resultant topological groupoid and $\Sh(\Xtt)$ for its topos of sheaves (normally we omit the topologies $\tau_0$ and $\tau_1$ and intuit these by context).
			\item Given a sheaf $(Y,q,\beta)$ over $\X$ and a subset $V \subseteq Y$, we will use $\overline{V}$ to denote the \emph{orbit} of $V$ under the action $\beta$, that is the subset
			\[
			\overline{V} = \lrset{y \in Y}{ \exists \, y' \in V \subseteq Y \text{ and } \alpha \in s^{-1}(q(y')) \subseteq X_1 \text{ such that } y = \beta(y',\alpha) } \subseteq Y.
			\]
		\end{enumerate}
	\end{nota}
	\begin{exs}
		We revisit the example topological groupoids from \cref{exs:example-grpds}, and describe their topoi of sheaves.
		\begin{enumerate}
			\item For a space $X$, the topos of sheaves on the categorically discrete groupoid on $X$ is the familiar topos of sheaves over the space.
			\item For a topologically discrete groupoid $\X$, its topos of sheaves is equivalent to the presheaf topos $\sets^\X$.
			\item The topos of sheaves on a topological group $G$ is the topos of discrete sets with a continuous $G$-action (as studied, for instance, in \cite[\S III.9]{SGL}), more commonly written as $\B G$.
			\item The topos of sheaves on a continuous group action on a space $G \times X \to X$ is the topos of $G$-equivariant sheaves over $X$, as studied in \cite{freyd} and \cite[Proposition A4.6]{SGL}.
		\end{enumerate}
	\end{exs}
	\begin{ex}\label{ex:definable_sheaves}
		A recurring example of a topological groupoid we will encounter is that induced by a groupoid of \emph{indexed} models for a geometric theory.  Let $\theory$ be a geometric theory over a signature $\Sigma$ and $M$ a model of $\theory$.  By an \emph{indexing} of $M$ by a set of parameters $\Index$, we mean an interpretation in $M$ of the expanded signature $\Sigma \cup \lrset{c_m}{m \in \Index'}$ where a constant symbol is added for each parameter in some subset $\Index' \subseteq \Index$, such that every element $n \in M$ is the interpretation of some $m \in \Index'$.  Equivalently, this is a choice of partial surjection $\Index \paronto M$.  We will abuse notation and not distinguish between a parameter and its interpretation in a model.

		Let $\X$ be a groupoid whose objects are models for $\theory$, each of which has a choice of $\Index$-indexing, and whose arrows are isomorphisms of these models.  Note that we do not require these isomorphisms to preserve the interpretation of the parameters.  Then $\X$ can be made into a topological groupoid as follows.
		\begin{enumerate}
			\item The set of objects $X_0$ is endowed with the topology whose basic opens are the \emph{sentences with parameters}:
			\[\classv{\phi}{m}_\X = \lrset{ M \in X_0}{M \vDash \phi(\vec{m})} \subseteq X_0,\]
			for $\phi$ a geometric formula and $\vec{m} \in \Index$ a tuple of parameters.
			\item The set of arrows $X_1$ is endowed with the topology generated by the basic opens
			\[
			\lrclass{
			\begin{matrix}
				\vec{m}_1 : \phi \\
				\vec{m}_2 \mapsto \vec{m}_3 \\
				\vec{m}_4 : \psi
			\end{matrix}
			}_\X
			=
			\lrset{M \xrightarrow{\alpha} N \in X_1}{
			\begin{matrix}
				M \vDash \phi(\vec{m}_1) \\
				\alpha(\vec{m}_2) = \vec{m}_3 \\
				N \vDash \psi(\vec{m}_4)
			\end{matrix},
			}
			\]
			where $\phi, \psi$ are formulae and $\vec{m}_1, \vec{m}_2, \vec{m}_3, \vec{m}_4 \in \Index$ are all tuples of parameters.  Given tuples of parameters $\vec{m},\vec{m}' \in \Index$, we will use the shorthand $\class{\vec{m} \mapsto \vec{m}'}_\X$ for the open subset $\lrset{\alpha \in X_1}{\alpha(\vec{m}) = \vec{m}'} \subseteq X_1$.
		\end{enumerate}
		Collectively, these topologies are known as the \emph{logical topologies}, given a choice of indexing $\Index \paronto \X$.
		
		The \emph{definable sheaves} over $\X$ form an important set of sheaves in $\Sh(\X)$.  For a geometric formula $\phi$ in context $\vec{x}$, consider the set
		\[
		\classv{\phi}{x}_\X = \lrset{\lrangle{\vec{n}, M}}{\vec{n} \in M \in X_0 \text{ and } M \vDash \phi(\vec{n})}.
		\]
		This has an evident projection $\pi_{\classv{\phi}{x}} \colon \classv{\phi}{x}_\X \to X_0$, given by $\lrangle{\vec{n},M} \mapsto M$, and an $X_1$-action where $\alpha \cdot \lrangle{\vec{n},M} = \lrangle{\alpha(\vec{n}),N}$ for an arrow $M \xrightarrow{\alpha} N \in X_1$.  When $\classv{\phi}{x}_\X$ is endowed with the topology whose basic opens are the subsets defined by formulae with parameters, i.e.\ subsets of the form
		\[\class{\vec{x},\vec{m}:\phi\land \psi}_\X = \lrset{\lrangle{\vec{n},M}}{M \vDash \phi(\vec{n}), M \vDash \psi(\vec{n},\vec{m})},\]
		for $\vec{m} \in \Index$ a tuple of parameters and $\psi$ another formula, then $\classv{\phi}{x}_\X$ becomes a sheaf for $\X$.
	\end{ex}

	There is a natural functor $\Sh(\X) \to \Sh(X_0)$ that forgets the $X_1$-action of a sheaf on $\X$.  This functor is the inverse image of a geometric morphism, which we label $u_\X$, that has many desirable properties:
	\begin{prop}[Theorem 3.6 \& Proposition 4.4 \cite{cont1}]\label{u_X-is-surj}
		For every topological groupoid $\X$, the geometric morphism $u_\X \colon \Sh(X_0) \to \Sh(\X)$ is surjective.  If $\X$ is an open topological groupoid, then $u_\X$ is open as well.
	\end{prop}
	
	\paragraph{Sheaves as a bi-functor.}
	The notion of the topos of sheaves on a topological groupoid is natural in the sense that there is a bi-functor
	\[
	\begin{tikzcd}
		\Sh \colon \TopGrpd \to \Topos.
	\end{tikzcd}
	\]
	Given a continuous functor between topological groupoids $\Phi \colon \X \to \Y$, the inverse image functor $\Sh(\Phi)^\ast \colon \Sh(\Y) \to \Sh(\X)$ sends a sheaf $W$ over $\Y$ to the pullback
	\[
	\begin{tikzcd}
			\Phi_0^\ast W \ar{r} \ar{d} & W \ar{d} \\
			X_0 \ar{r}{\Phi_0} & Y_0,
	\end{tikzcd}
	\]
	while the $X_1$-action on $\Phi_0^\ast W$ is given by $g \cdot (s(g),w) = (t(g),\Phi_1(g) \cdot w)$.  Given a continuous transformation $a \colon \Phi \Rightarrow \Psi$, the induced natural transformation $\Sh(a) \colon \Sh(\Phi)^\ast \Rightarrow \Sh(\Psi)^\ast$ has as its component at a sheaf $W \in \Sh(\Y)$ the morphism of sheaves $\Sh(a)_W \colon \Phi_0^\ast W \to \Psi_0^\ast W$ that acts by $(x,w) \mapsto (x,a_x \cdot w)$.
	
	\begin{ex}
		Let $\X = (X_1 \rightrightarrows X_0)$ be a topological groupoid.  The space of objects $X_0$ yields a categorically discrete topological groupoid, and the obvious (identity on objects) functor $u \colon X_0 \to \X$ is continuous.  When we apply the bi-functor $\Sh$ to $u$, we return the geometric morphism $u_\X \colon \Sh(X_0) \to \Sh(\X)$ from \cref{u_X-is-surj}.
	\end{ex}
	
	\subsection{Representing groupoids of classifying topoi}
	
	Throughout this paper, we will exploit the interplay between classifying topoi and topological grouopoids in order to prove our results.  Recall from \cite[Part D]{elephant} that a classifying topos for a (geometric) theory $\theory$ is a topos $\topos_\theory$ such that the representable pseudo-functor $\Topos(-,\topos_\theory) \colon \Topos\op \to \CAT$ is naturally isomorphic to the pseudo-functor $\Tmodels{-} \colon \Topos\op \to \CAT $ that sends a topos $\ftopos$ to the category of models of $\theory$ internal to $\ftopos$.  This universal property defines a classifying topos up to equivalence.  This is an exhaustive perspective on topos theory because every {geometric} theory has a classifying topos (\cite[Theorem 2.1.10]{TST}) and every topos classifies some geometric theory (\cite[Theorem 2.1.11]{TST}).
	
	The classifying topos contains a special internal model of $\theory$, named the \emph{universal model} $U_\theory$, that corresponds to the geometric morphism $\id_{\topos_\theory} \in \Topos(\topos_\theory,\topos_\theory) \simeq \Tmodels{\topos_\theory}$ (see \cite[\S 2.1.2]{TST} or \cite[\S D3.1]{elephant}).  The model $U_\theory$ is witnessed by the inclusion of the \emph{syntactic category} as a full subcategory $\cat_\theory \subseteq \topos_\theory$.  We denote the interpretation of a geometric formula $\phi$ in $U_\theory$, i.e.\ an object of the subcategory $\cat_\theory \subseteq \topos_\theory$, by $\form{\phi}{x}_\theory$.  The objects of $\cat_\theory$ \emph{generate} $\topos_\theory$ and so any geometric morphism $f \colon \ftopos \to \topos_\theory$ is determined by the image of $\form{\phi}{x}_\theory$ under the inverse image functor $f^\ast$.  Indeed, $f^\ast \form{\phi}{x}_\theory$ describes the interpretation of $\phi$ in the $\theory$-model in $\ftopos$ corresponding to $f \colon \ftopos \to \topos_\theory$; it is in this sense that the model $U_\theory$ is universal.
	
	\begin{rem}
		The classifying topos of a theory $\theory$ has enough points if and only if $\theory$ is \emph{complete} with respect to its set-based models, i.e.\ if two formulae $\phi$ and $\psi$ are interpreted identically in every set-based model of $\theory$, then $\phi$ and $\psi$ are provably equivalent in the theory $\theory$.
	\end{rem}
	
	\begin{df}
		A topological groupoid $\X$ is said to represent a topos $\topos$ (or a theory $\theory$) if there is an equivalence of topoi $\Sh(\X) \simeq \topos$ (or $\Sh(\X) \simeq \topos_\theory$). 
	\end{df}
	
	Thus, the theorem of Butz and Moerdijk \cite{BM} asserts that every topos with enough points has at least one representing groupoid.  A classification of the representing open topological groupoids for a given theory is provided in \cite{myselfgrpds}.
	\begin{prop}[Corollary 3.9 \cite{myselfgrpds}]\label{class-repr-grpds-for-theory}
		Let $\theory$ be a geometric theory and let $\X = (X_1 \rightrightarrows X_0)$ be a (small) groupoid.  The following are equivalent:
		\begin{enumerate}
			\item there exist topologies, satisfying the $T_0$ separation axiom, on $X_1$ and $X_0$ making $\X$ an \emph{open} topological groupoid for which there is an equivalence $\Sh(\X) \simeq \topos_\theory$;
			
			\item the groupoid $\X$ is a groupoid of set-based models of $\theory$, i.e.\ there is an inclusion functor $\X \subseteq \Tmodels{\sets}$, such that each model $M \in X_0$ can be given an \emph{indexing} by a set of parameters $\Index$, as in \cref{ex:definable_sheaves}, such that:
			\begin{enumerate}
				\item the set $X_0$ of models is \emph{conservative}, i.e.\ if two formulae have identical interpretations in every model of $X_0$, then they are $\theory$-provably equivalent;
				\item the groupoid $\X$ \emph{eliminates parameters}, meaning that for every tuple of parameters $\vec{m} \in \Index$, the orbit of the interpretations of the tuple $\vec{m}$ in the groupoid is \emph{definable} by a formula $\phi$ without parameters, that is to say 
				\begin{align*}
					\overline{\class{\vec{x} = \vec{m}}}_\X & = \lrset{\lrangle{\vec{n},N}}{\exists \, M \xrightarrow{\alpha} N \in X_1 \text{ with } M \vDash \exists \, \vec{x} \ \vec{m} = \vec{x} \text{ and } \alpha(\vec{m}) = \vec{n} } \\
					& = \lrset{\lrangle{\vec{n},N}}{N \in X_0, \, N \vDash \phi(\vec{n})}_\X \\
					&= \classv{\phi}{x}_\X.
				\end{align*}
			\end{enumerate}
		\end{enumerate}
	\end{prop}
	
	If $\X$ is a groupoid of set-based models of a geometric theory $\theory$ satisfying the conditions of \cref{class-repr-grpds-for-theory}, the topologies on $X_0$ and $X_1$ for which $\X$ is an open topological groupoid with $\Sh(\X) \simeq \topos_\theory$ are, unsurprisingly, the {logical topologies} from \cref{ex:definable_sheaves}.  Under the equivalence $\Sh(\X) \simeq \topos_\theory$, the definable sheaf $\classv{\phi}{x}_\X \in \Sh(\X)$ from \cref{ex:definable_sheaves} is identified with the object $\form{\phi}{x}_\theory \in \topos_\theory$.
	
	\begin{exs}\label{ex:elim_para}
		We give some examples of groupoids of models that eliminate parameters.
		\begin{enumerate}
			\item For a fixed infinite set of parameters $\Index$, the groupoid of \emph{all} $\Index$-indexed models of a geometric theory $\theory$ (and all their isomorphisms) eliminates parameters (\cite[Corollary 8.25]{myselfgrpds}), namely we have that
			\[
			\overline{\class{\vec{x} = \vec{m}}}_\X = \lrclass{\vec{x} : \bigwedge_{m_i = m_j} x_i = x_j}_\X.
			\]
			Thus, if the set of $\Index$-indexed models is conservative, the theory $\theory$ is represented the topological groupoid of all $\Index$-indexed groupoids endowed with the logical topologies (cf.\ \cite{forssellphd,forssell,awodeyforssell}).
			
			\item Let $M \vDash \theory$ be a model where the (model-theoretic) type of each tuple $\vec{m} \in M$ is \emph{isolated} by a formula $\chi_{\vec{m}}$, i.e.\
			\[
			\tp_M(\vec{m}) = \lrset{\phi}{M \vDash \phi(\vec{m})} = \lrset{\phi}{\theory \text{ proves the sequent } \chi_{\vec{m}} \vdash_{\vec{x}} \phi} .
			\]
			Then the automorphism group $\Aut(M)$ eliminates parameters (for the indexing of $M$ by its own elements) if and only if $M$ is \emph{ultrahomogeneous}, i.e.\ if two tuples $\vec{n},\vec{m} \in M$ have the same type then there is an automorphism of $M$ sending $\vec{n}$ pointwise onto $\vec{m}$ (\cite[Lemmas 8.4 \& 8.5]{myselfgrpds}).  In this case, we have that
			\[
			\overline{\class{\vec{x} = \vec{m}}}_{\Aut(M)} = \classv{\chi_{\vec{m}}}{x}_{\Aut(M)}.
			\]
			Thus, if $M$ is a conservative model for $\theory$, then when $\Aut(M)$ is endowed with topology generated by the subsets 
			\[
			\class{\vec{m} \mapsto \vec{m}'}_{\Aut(M)} = \lrset{\alpha \in \Aut(M)}{\alpha(\vec{m}) = \vec{m}'},
			\]			
			for finite tuples $\vec{m}, \vec{m}' \in M$, there is an equivalence of topoi $\topos_\theory \simeq \B \Aut(M)$ (cf.\ \cite[Theorem 3.1]{caramellogalois}).
		\end{enumerate}
	\end{exs}
	
	\begin{lem}(cf.\ Scholium C5.2.6 \cite{elephant})\label{every-set-of-models-has-expansion}
		Let $\theory$ be a theory with enough points.  Given any set $W$ of models of $\theory$, there exists a representing groupoid $\X = (X_1 \rightrightarrows X_0)$ of $\theory$ where $W \subseteq X_0$.
	\end{lem}
	\begin{proof}
		Let $\Index$ be a cardinal large enough that every model $M \in W$ admits a $\Index$-indexing and that the set of all $\Index$-indexed models of $\theory$ are conservative.  Then the groupoid $\X$ of all $\Index$-indexed models of $\theory$ is a representing groupoid for $\theory$ for which $W \subseteq X_0$.
	\end{proof}
	
	\subsection{Categories of fractions}

	For the reader unfamiliar with categories of fractions, we briefly recall their construction here.  The localisation of a category was introduced by Gabriel and Zisman \cite{gabrielzisman} and later extended to the bi-categorical setting by Pronk \cite{pronk}.  Let $\cat$ be a category and $\Sigma$ a class of morphisms in $\cat$.  The \emph{localisation} of $\cat$ by $\Sigma$, if it exists, is a functor $\cat \to \cat_\Sigma$ such that, for any functor $F \colon \cat \to \dcat$ where $F$ sends each $f \in \Sigma$ to an isomorphism in $\dcat$, there is a unique extension
	\[
	\begin{tikzcd}
		\cat \ar{r} \ar{rd}[']{F} & \cat_\Sigma \ar[dashed]{d} \\
		& \dcat.
	\end{tikzcd}
	\]
	In general, there is no guarantee that $\cat_\Sigma$ exists, or even if it does, that it admits a compact description.  The exception is when $\Sigma$ forms either a \emph{left} (or a \emph{right}) \emph{calculus of fractions}.  That is to say, $\Sigma$ is wide, satisfies the left (respectively, right) Ore condition and the left (resp., right) cancellability condition.
	
	In this case, the localisation exists, and can easily be described.  If $\Sigma$ is a left calculus of fractions on $\cat$, we denote its localisation by $[\Sigma^{-1}]\cat$, and if $\Sigma'$ is a right calculus of fractions, we denote its localisation by $\cat[\Sigma'^{-1}]$.
	\begin{enumerate}
		\item The categories $\cat$, $[\Sigma^{-1}]\cat$ and $\cat[\Sigma'^{-1}]$ all share the same objects.
		\item An arrow $(f,g) \colon a \to b$ in $[\Sigma^{-1}]\cat$ is described by a cospan
		\[
		\begin{tikzcd}
			& b \ar{d}{g} \\
			a \ar{r}{f} & x
		\end{tikzcd}
		\]
		where $g \in \Sigma$.  The fact that $\Sigma$ is a left calculus of fractions entails that, given composable arrows $(f,g) \colon a \to b, (f',g') \colon b \to c$ in $\cat[\Sigma^{-1}]$, there is an essentially unique way to complete the diagram
		\[
		\begin{tikzcd}
			&& c \ar{d}{g'} \\
			& b \ar{d}[']{g} \ar{r}{f'} & y \ar[dashed]{d}{k} \\
			a \ar{r}{f} & x \ar[dashed]{r}{h} & z
		\end{tikzcd}
		\]
		for which $k \in \Sigma$, thus yielding the composite $(f',g') \circ (f,g)$ in $[\Sigma^{-1}]\cat$.
		
		In $\cat[\Sigma'^{-1}]$, an arrow $(f,g) \colon a \to b$ is instead given by a span
		\[
		\begin{tikzcd}
			x \ar{r}{g} \ar{d}[']{f} & b \\
			a &
		\end{tikzcd}
		\]
		for which $f \in \Sigma'$.
	\end{enumerate}
	
	When $\cat$ is a bi-category and $\Sigma$ is a class of 1-morphisms, the same story unfolds for a bi-categorical localisation, a bi-functor $\cat \to \cat_\Sigma$ universal amongst all bi-functors that send morphisms $f \in \Sigma$ to equivalences.  If $\Sigma$ forms a left bi-calculus of fractions, then the localization is given by a \emph{bi-category of fractions} $[\Sigma^{-1}]\cat$, whose objects and 1-morphisms are identical to the 1-categorical case.  A 2-cell $\alpha \colon (f,g) \Rightarrow (f',g')$ in $[\Sigma^{-1}]\cat$ is simply a 2-cell $\alpha \colon f \Rightarrow f'$ in $\cat$.  For the compositions of 2-cells in $[\Sigma^{-1}]\cat$, the reader is directed to \cite{pronk}.

	\section{A right calculus of fractions is impossible}\label{sec:rightcalc}
	
	Before embarking on our proof of \cref{thm:main-thm-intro}, we justify why we consider a \emph{left} (bi-)calculus of fractions on topological groupoids by demonstrating that a (bi-)equivalence with a \emph{right} (bi-)calculus of fractions is impossible.
	
	\paragraph{Why make the distinction?}
	The localisation of a category by a calculus of fractions is defined by a universal property (\cite[Theorem 21]{pronk}), and so some readers may reasonably wonder why we are emphasising the distinction between a left calculus and a right calculus.  The importance arises by considering when two objects are isomorphic in the localisation.
	
		If $\cat$ is a category with a right calculus of fractions $\Sigma$, a pair of objects $c,d \in \cat$ are isomorphic in the localisation $\cat[\Sigma^{-1}]$ if and only if there is a \emph{span} of morphisms $c \xleftarrow{a} e \xrightarrow{b} d$ where $a, b \in \Sigma$.  Conversely, If $\dcat$ is a category with a left calculus of fractions $\Sigma'$, a pair of objects $c,d \in \dcat$ are isomorphic in the localisation $[\Sigma'^{-1}]\dcat$ if and only if there is a \emph{cospan} of morphisms $c \xrightarrow{a} e \xleftarrow{b} d$ where $a, b \in \Sigma'$.
	
	Thus, the condition of Morita equivalence for localic groupoids derivable from Moerdijk's equivalence \cite[Theorem 7.7]{cont1} is vastly different from the condition we will obtain for topological groupoids in \cref{morita-eqv-for-ectgrpd}.
	
	\begin{prop}\label{prop:no_right_calc_for_top}
		For any bi-subcategory $\cat \subseteq \TopGrpd$, and any bi-calculus of right fractions $\Sigma$ on $\cat$, the restricted bi-functor $\Sh \colon \cat \to \Topos_\WEP^\iso$ cannot induce a bi-equivalence $\Topos_\WEP^\iso  \simeq \cat[\Sigma^{-1}]$.
	\end{prop}
	\begin{proof}
		We construct an example of a geometric morphism that cannot be obtained by a span of continuous functors between topological groupoids.  Let $\X = (X_1 \rightrightarrows X_0) $ be a topological groupoid contained in $\cat \subseteq \TopGrpd$ such that the topos $\Sh(\X)$ has a point $p \colon \sets \to \Sh(\X)$ that does not correspond to a point of $X_0$, i.e.\ there is no factorisation
		\[
		\begin{tikzcd}
			& \Sh(X_0) \ar{d}{u_\X} \\
			\sets \ar[dashed]{ru}[marking]{/} \ar{r}[']{p} & \Sh(\X).
		\end{tikzcd}
		\]
		For example, $\Sh(\X)$ could be the classifying topos for a theory with enough points and unboundedly many models (the theory of groups, for instance), which ensures the existence of such a $p$ for any $X_0$.

		Suppose that there is a bi-equivalence $\Topos_\WEP^\iso \simeq \cat[\Sigma^{-1}]$.  Then there is a continuous functor of topological groupoids $\Phi \colon \Y \to \X \in \cat$ such that
		\[
		\begin{tikzcd}
			\Sh(\Y) \simeq \sets \ar{rr}{\Sh(\Phi) \simeq p} && \Sh(\X).
		\end{tikzcd}
		\]
		Consequently, there is a commutative square of geometric morphisms
		\[
		\begin{tikzcd}
			\Sh(Y_0) \ar{d}[']{u_\Y} \ar{rr}{\Sh(\Phi_0)} &[-30pt] & \Sh(X_0) \ar{d}{u_\X} \\
			\Sh(\Y) & \simeq \sets \ar{r}{p} & \Sh(\X).
		\end{tikzcd}
		\]
		Any point of the space $Y_0$ yields a section of $u_\Y$, as in the diagram
		\[
		\begin{tikzcd}
			\Sh(Y_0) \ar{d}[']{u_\Y} \ar{rr}{\Sh(\Phi_0)} &[-30pt] & \Sh(X_0) \ar{d}{u_\X} \\
			\Sh(\Y) & \simeq \sets \ar[dashed, bend right]{lu}  \ar{r}{p} & \Sh(\X).
		\end{tikzcd}
		\]
		But such a point would yield a factorisation of $p$ through $u_\X$, a contradiction.  So we conclude that $Y_0$ is the empty space.  But then there would be a surjective geometric morphism
		\[\begin{tikzcd}
			\0_\Topos \simeq \Sh(Y_0) \ar{r}{u_\Y} & \Sh(\Y) \simeq \sets.
		\end{tikzcd}\]
		Such a surjection would yield an equivalence $\0_\Topos \simeq \sets$ (because the unique geometric morphism $\0_\topos \to \sets$ is already a geometric embedding), an obvious contradiction\footnote{Of course, assuming the chosen model of set theory is consistent.}, and so we conclude that $\Topos_\WEP^\iso \not \simeq \cat[\Sigma^{-1}]$ as desired.
	\end{proof}
	
	\begin{rem}
		If we restrict to certain bi-subcategories $\dcat \subseteq \Topos_\wep^\iso$, we are then able to obtain a bi-equivalence $\dcat \simeq \cat[\Sigma^{-1}]$ for $\cat \subseteq \TopGrpd$ and a right bi-calculus of fractions on $\cat$.  For example, Pronk establishes in \cite[Theorem 27]{pronk} a bi-equivalence
		\[
		\text{\bf \'{E}tendue}^\iso_{\text{sp}} \simeq \text{\bf \'{E}taleGrpd}[W^{-1}]
		\]
		between the bi-category of (spatial) \emph{étendues} and a localisation of the bi-category of \emph{étale} topological groupoids, i.e.\ those groupoids whose source and target maps are étale/local homeomorphisms, using a right bi-calculus of fractions.  The obstacle identified in \cref{prop:no_right_calc_for_top} is not encountered in the \'etendue case because an \'etendue does not have unboundedly many points.
	\end{rem}

	\section{Logical groupoids}\label{sec:loggrpd}
	
	The eventual bi-equivalence result we will prove involves \emph{étale complete topological groupoids}.  It is of interest to split up the definition of étale completeness for topological groupoids into a condition on the topology on the set of arrows (see \cref{df:loggrpd} below), and a condition on the existence of these arrows (cf.\ \cref{etale_complete_iff_logical+stuff}), as much of the theory we develop requires only the former.  Before introducing these \emph{logical groupoids}, we give the special case for groups to justify that this is a natural topological condition to impose.
	
	\begin{prop}\label{prop:eqv-conditions-for-log-group}
		For $G$ a topological group satisfying the $T_0$ separation axiom, the following are equivalent:
		\begin{enumerate}
			\item\label{enum:open-subgrp-basis} the open subgroups of $G$ form a basis of open neighbourhoods for the identity element;
			\item\label{enum:coarsest-topology} the topology $\tau$ on $G$ is the coarsest topology determined by the topos $\B G^\tau$ -- that is, for any other topology $\sigma \subseteq \tau$ on $G$ making $G$ a topological group, if the evident inclusion functor $\B G^\sigma \to \B G^\tau$ witnesses an equivalence $\B G^\sigma \simeq \B G^\tau$, then $\tau = \sigma$;
			\item\label{enum:subgrp-of-permutation} there is a set $X$ such that $G$ is a topological subgroup of the permutation group $\Omega(X)$, endowed with the `pointwise convergence' topology, i.e.\ the topology where the basic open subgroups
			\[
			\lrset{\alpha \colon X \xrightarrow{\sim} X}{\alpha(\vec{x}) = \vec{x}},
			\]		
			for finite tuples $\vec{x} \in X$, form a basis of open neighbourhoods of the identity element.
		\end{enumerate}
	\end{prop}
	\begin{proof}
		One implication, \cref{enum:subgrp-of-permutation} $\implies$ \cref{enum:open-subgrp-basis}, is immediate.  The equivalence \cref{enum:open-subgrp-basis} $\iff$ \cref{enum:coarsest-topology} effectively asserts that the topos $\B G^\tau$ is determined by the open subgroups of $G$, which follows from the construction of a site for $\B G^\tau$ given in \cite[\S III.9]{SGL}.  To prove \cref{enum:coarsest-topology} $\implies$ \cref{enum:subgrp-of-permutation}, let $\theory$ be a (single-sorted) theory classified by $\B G$.  Then, by \cref{class-repr-grpds-for-theory}, the group $G$ is the automorphism group of a model of $\theory$, endowed with the pointwise convergence topology.  Taking $X$ as the underlying set of the model, the result now follows.
	\end{proof}

	\begin{df}\label{df:loggrpd}
		\begin{enumerate}
			\item A topological group satisfying one of the equivalent conditions in \cref{prop:eqv-conditions-for-log-group} is said to be \emph{logical}.
			
			\item A topological groupoid $\Xtt = (X_1^{\tau_1} \rightrightarrows X_0^{\tau_0})$ is said to be a \emph{logical groupoid} if:
			\begin{enumerate}
				\item the groupoid $\Xtt$ is open;
				\item both $X_0^{\tau_0}$ and $X_1^{\tau_1}$ are $T_0$ spaces;
				\item for any other topology $\sigma \subseteq \tau_1$ on $X_1$, whenever the evident inclusion $\Sh\left(\X_{\tau_0}^{\sigma}\right) \to \Sh\left(\Xtt\right)$ witnesses an equivalence $\Sh\left(\X_{\tau_0}^{\sigma}\right) \simeq \Sh\left(\Xtt\right)$, then $\tau_1 = \sigma$.  In other words, $\tau_1$ is the coarsest topology on $X_1$ determined by the topos $\Sh(\Xtt)$ (and so logical groups are logical groupoids).
			\end{enumerate}
		\end{enumerate}
	\end{df}

	\begin{ex}
		Let $\theory$ be a geometric theory, and let $\X$ be a groupoid of $\theory$-models indexed by a set of parameters $\Index$ such that $\X$ eliminates parameters.  The resultant topological groupoid $\X$, where the logical topologies from \cref{ex:definable_sheaves} have been endowed, is a logical groupoid by \cite[Proposition 6.4 \& Lemma 7.3]{myselfgrpds}.
	\end{ex}
	
	The terminology `logical groupoid' has been chosen because the above construction suffices to describe all logical groupoids: 
	
	\begin{prop}\label{class-of-loggrpd}
		An open topological groupoid $\Xtt = \left(X_1^{\tau_1} \rightrightarrows X_0^{\tau_0}\right)$ is a logical groupoid if and only if there exists a signature $\Sigma$ such that
		\begin{enumerate}
			\item $\X$ is a groupoid of set-based $\Sigma$-structures,
			\item there is an indexing of these models by a set of parameters $\Index$ for which the constituent topologies $\tau_1, \tau_0$ on $\X$ are the logical topologies for this indexing (see \cref{ex:definable_sheaves}).
		\end{enumerate}
	\end{prop}
	\begin{proof}
		This characterisation of logical groupoids is a synthesis of results from \cite{myselfgrpds}.  If $\X$ is a groupoid of indexed $\Sigma$-structures, then resultant topological groupoid given by endowing the logical topologies is logical by \cite[Proposition 6.4]{myselfgrpds}.  
		
		Conversely, suppose that $\Xtt$ is logical, and let $\theory$ be a geometric theory classified by $\Sh(\Xtt)$.  Moreover, we can assume that $\theory$ is \emph{inhabited}, i.e.\ $\theory$ proves the sequent $\top \vdash \exists x \, \top$ -- the idea is the following: every geometric theory $\theory'$ is equivalent to one with a non-empty `dummy sort', i.e.\ a sort with one constant symbol $c$ and the axiom $\top \vdash_x x = c$, which can then be amalgamated into any uninhabited sort of $\theory'$ using the techniques of \cite[Lemma D1.4.13]{elephant} (this also follows from \cite[\S VII.3]{JT}).  Then by \cref{class-repr-grpds-for-theory}, $\X$ can be identified with a groupoid of indexed $\theory$-models, i.e.\ a groupoid of indexed $\Sigma$-structures where $\Sigma$ is the underlying signature of $\theory$.  By \cite[Remark 5.5]{myselfgrpds}, the induced logical topology for objects is precisely $\tau_0$.  By \cite[Proposition 6.4]{myselfgrpds}, the topology $\tau_1$ on the space of arrows contains the logical topology for arrows, while the reverse inclusion follows as $\X$ is logical.
	\end{proof}

	The restriction to logical groupoids is not prohibitive by the following observation:
	
	\begin{coro}
		For every open topological groupoid $\Xtt = (X_1^{\tau_1} \rightrightarrows X_0^{\tau_0})$ in which $\tau_0$ and $\tau_1$ are both $T_0$ topologies, by choosing a new topology $\sigma $ on $X_1$, we obtain a logical groupoid $\X_{\tau_0}^{\sigma} = (X_1^{\sigma} \rightrightarrows X_0^{\tau_0})$ for which $\Sh\left(\X_{\tau_0}^{\sigma}\right) \simeq \Sh\left(\Xtt\right)$.
	\end{coro}
	\begin{proof}
		Let $\theory$ be an inhabited geometric theory classified by $\Sh(\Xtt)$.  Then, by \cref{class-repr-grpds-for-theory}, $\X$ can be identified with a groupoid of indexed $\theory$-models that is both conservative and eliminates parameters.  When given the logical topologies, the resultant topological groupoid is open by \cite[Lemma 7.3]{myselfgrpds} and logical by \cref{class-of-loggrpd} (and moreover the logical topology on objects is the topology $\tau_0$ by \cite[Remark 5.5]{myselfgrpds}).
	\end{proof}
	
	\begin{df}
		Let $\X$ be topological groupoid.  A (\emph{topological}) \emph{subgroupoid} $\Y$ of $\X$ consists of a subgroupoid
		\[
		\begin{tikzcd}
			Y_1 \times_{Y_0} Y_1 \ar{r} \ar[hook]{d} &	Y_1 \ar[shift left=2]{r} \ar[shift right=2]{r} \ar[hook]{d} & Y_0 \ar[hook]{d} \ar{l} \\
			X_1 \times_{X_0} X_1 \ar{r} &	X_1 \ar[shift left=2]{r} \ar[shift right=2]{r} & X_0 \ar{l}
		\end{tikzcd}
		\]
		where $Y_1 \hookrightarrow X_1$ and $Y_0 \hookrightarrow X_0$ are subspace inclusions.  Such a groupoid is defined by any subset $Y_1 \subseteq X_1$ that is closed under composition and inverses as the space of objects can be recovered as $s(Y_1) = t(Y_1)$.
	\end{df}
	
	\begin{coro}
		If $\X$ is a logical groupoid, then any (topological) subgroupoid $\Y \hookrightarrow \X$ is a logical groupoid too.
	\end{coro}
	\begin{proof}
		First note that, being a subgroupoid of an open topological groupoid, $\Y$ is also open, and that $Y_0$ and $Y_1$ also inherit the $T_0$ separation axiom from $X_0$ and $X_1$.  By \cref{class-of-loggrpd}, $\X$ can be identified with a groupoid of indexed set-based structures for some signature $\Sigma$, and the topologies on $\X$ are the logical topologies for this indexing.  Thus, $\Y$ is also identified with a groupoid of indexed $\Sigma$-structures, and the topologies on $\Y$ are induced logical topologies.  Hence, by \cref{class-of-loggrpd}, $\Y$ is logical.
	\end{proof}
	
	\section{Weak equivalences of logical groupoids}\label{sec:weqv}
	
	In this section, we identify the class of \emph{weak equivalences} $\weqv$ of logical groupoids.  We begin by stating their classification, whose proof is the primary aim of the remainder of this section.
	
	\begin{dfs}
	\begin{enumerate}
		\item   Let $\X$ be a topological groupoid and $\Y \subseteq \X$ a subgroupoid.  If the inclusions $Y_1 \subseteq X_1$ and $Y_0 \subseteq X_0$ are the inclusions of {open} subspaces, we say that $\Y$ is an \emph{open subgroupoid} of $\X$.
		
		\item For any groupoid $\X$ and a pair of subgroupoids $\Y, \W \subseteq \X$, a subset $V \subseteq X_1$ of arrows canonically admits a $Y_1$-action and a $W_1$-action given, respectively, by precomposition and postcomposition, i.e.\ for $\beta \in Y_1$ and $\gamma \in W_1$,
		\[
		\left(x \xrightarrow{\alpha} y\right) \cdot \beta = \left(x' \xrightarrow{\alpha \circ \beta} y \right), \quad \gamma \cdot \left(x \xrightarrow{\alpha} y \right) = \left(x \xrightarrow{\gamma \circ \alpha} y'\right).
		\]
		The \emph{bi-orbit} of $x \xrightarrow{\alpha} y \in V$ is then the set $\lrset{\eta \in V}{\exists \beta \in Y_1, \, \exists \gamma \in W_1 \text{ such that } \eta = \gamma \circ \alpha \circ \beta}$, which, breaking from our previous notation in  \cref{notation-for-sheaves}, we denote as ${ _\W[\alpha]_\Y}$.  We will use ${ _\W[V]_\Y}$ to denote the set of bi-orbits of elements of $V$.  Note that $_\W[V]_\Y$ is evidently a subquotient of $X_1$, i.e.\ there is a diagram of maps
		\[
		\begin{tikzcd}
			X_1 & \ar[hook']{l} V \ar[two heads]{r} & _\W[V]_\Y.
		\end{tikzcd}
		\]
		Therefore, if $\X$ is moreover a topological groupoid, then $_\W[V]_\Y$ can be endowed with the (sub)quotient topology.
		
		\item (Définition 10.2.2 \cite{EGAIV}) Recall that a continuous map $f \colon X \to Y$ between spaces is a \emph{quasi-homeomorphism} if the inverse image map on open subsets $f^{-1} \colon \opens(Y) \to \opens(X)$ is bijective (equivalently, if the inverse image map $f^{-1} \colon \mathcal{C}(Y) \to \mathcal{C}(X)$ on closed subsets is a bijection).
	\end{enumerate}	
	\end{dfs}
	
	\begin{thm}\label{classification-of-weqv}
		Let $\X$ be a logical groupoid and let $i \colon \Y \hookrightarrow \X$ be a subgroupoid.  The inclusion $i \colon \Y \hookrightarrow \X$ induces an equivalence of sheaf topoi $\Sh(i) \colon \Sh(\Y) \xrightarrow{\sim} \Sh(\X)$ if and only if, for every open subgroupoid $\U = (U_1 \rightrightarrows U_0)$ of $\X$, the continuous map
		\begin{align*}
			\iota \colon  { _\Y[s^{-1}(U_0) \cap t^{-1}(Y_0)]_\U} & \to { _\X[s^{-1}(U_0)]_\U  }, \\
			{ _\Y[\alpha]_\U } & \mapsto { _\X[\alpha]_\U }
		\end{align*}
		is a {quasi-homeomorphism}.
	\end{thm}
	
	We spell out some of the notation: $s^{-1}(U_0) \cap t^{-1}(Y_0)$ is the set of all arrows $x \xrightarrow{\alpha} y \in X_1$ with $x \in U_0$ and $y \in Y_0$.  As $\U$ is a groupoid, the composite $\alpha \circ \beta$ of any such $\alpha $ with an arrow $\beta \in U_1$ remains an arrow whose source is contained in $U_0$.  Likewise, post-composition with arrows in $Y_1$ keeps the target in $Y_0$.  Thus, the quotient space ${ _\Y[s^{-1}(U_0) \cap t^{-1}(Y_0)]_\U}$ is easily formed, as is ${ _\X[s^{-1}(U_0)]_\U}$.

	\begin{df}
	An inclusion $i \colon \Y \hookrightarrow \X$ satisfying the hypotheses of \cref{classification-of-weqv} is said to be a \emph{weak equivalence}.  We denote the class of weak equivalences of logical groupoids by $\weqv$.
	\end{df}
	
	\begin{coro}\label{weqv-is-wide}
		The class $\weqv$ of weak equivalences is wide, i.e.\ it contains all identities and is closed under composition and continuous isomorphisms.
	\end{coro}
	\begin{proof}
		This is a direct consequence of \cref{classification-of-weqv}, since the class of inclusions of topological subgroupoids that is sent by $\Sh$ to equivalences of topoi must evidently be wide.  A more concrete demonstration of this fact is contained in \cref{skula-dense-orbits-is-wide} and \cref{thick-is-wide} below.
	\end{proof}

	The remainder of this section is devoted to proving \cref{classification-of-weqv} and studying its consequences.  The proof of \cref{classification-of-weqv} is divided into two steps:
	\begin{enumerate}[label =  (\arabic*)]
		\item First, we demonstrate in \cref{incl-of-grpd-induces-localic} that the inclusion of a subgroupoid $i \colon \Y \hookrightarrow \X$ of a logical groupoid yields a \emph{relational extension}, also called a \emph{localic extension} in \cite[Definition 7.1.1]{TST}, of any geometric theory classified by $\Sh(\X)$, and so $\Sh(i)$ is a \emph{localic} geometric morphism.  This extends the analogous classical observation in model theory for topological automorphism groups (see \cite[\S 4]{hodges}). 
		
		\item Since $\Sh(i)$ is localic, the problem of classifying when $\Sh(i)$ is an equivalence reduces to checking for an isomorphism on the level of subobject lattices.  By considering subobjects of the \emph{Moerdijk generators} for the topos $\Sh(\X)$, identified in \cite[\S 6]{cont1}, we derive the statement of \cref{classification-of-weqv}.  We will also show that, in the special case where the codomain is a logical group, the conditions of \cref{classification-of-weqv} simplify and the weak equivalences whose codomain is a logical group are precisely the inclusions of dense subgroups (\cref{groups-thick-becomes-dense}).
	\end{enumerate}
	The remainder of this section then proceeds as follows:
	\begin{enumerate}[label =  (\arabic*)]
		\setcounter{enumi}{2}
		\item We next identify two weakenings of \cref{classification-of-weqv} -- instead of asking when $\Sh(i)$ is an equivalence, we study when $\Sh(i)$ is a \emph{surjective} geometric morphism (\cref{surjection-iff-skula-dense-orbits}) and when it is a \emph{subtopos inclusion} (\cref{subtopos-iff-thick}).  Combining both will give an alternative, equivalent description of the weak equivalences $\weqv$ (\cref{weqv-skula-dense-orbits-and-thick}).  We also recover in \cref{full-replete-yields-subtopos} the primary result of \cite{forssell-subgroupoid} that the inclusion of a full \emph{replete} subgroupoid yields a subtopos.
		
		\item Using \cref{surjection-iff-skula-dense-orbits} and \cref{full-replete-yields-subtopos}, we then prove in \cref{surj-incl-fact-is-full-ess-image} that, for a continuous functor $\Phi \colon \X \to \Y$ between logical groupoids, the surjection-inclusion factorisation of $\Sh(\Phi)$ is induced by the \emph{full essential image} factorisation of $\Phi$.
		
		\item Finally, in \cref{geo_is_hom_of_grpd} we demonstrate that, given two logical groupoids $\X, \Y$, every geometric morphism $\Sh(\X) \to \Sh(\Y)$ is induced by a cospan of continuous functors 
		\[
		\begin{tikzcd}
			\X \ar{r}{\Phi} & \W & \ar[hook']{l}[']{\Psi} \Y,
		\end{tikzcd}
		\]
		where $\Psi$ is a weak equivalence.
	\end{enumerate}

	\subsection{Relational extensions}

	The first step in the proof of \cref{classification-of-weqv} is to demonstrate that inclusions of subgroupoids of logical groupoids induce \emph{localic} geometric morphisms.  Our proof is reminiscent of the classical model-theoretic argument that subgroups of the automorphism group of a structure induce relational extensions, as can be found in \cite[Theorem 4.14]{hodges}.  Namely, let $X$ be a set and let $\Omega(X)$ denote the topological permutation group of $X$.  Each subgroup $G \subseteq \Omega(X)$ induces a relational structure $\Sigma$ on the set $X$, and $G \cong \Aut_\Sigma(X)$ if and only if $G \subseteq \Omega(X)$ is closed.  Moreover, given subgroup inclusions $H \subseteq G \subseteq \Omega(X)$, the relational structure induced by $H$ on $X$ is an expansion of the structure induced by $G$.  Thus, the fact that subgroupoids yield localic geometric morphisms is not too surprising, once we recall that relational extensions of theories induce localic geometric morphisms between their classifying topoi \cite[\S 7]{TST}.
	
	\begin{df}
		Recall that a geometric morphism $f \colon \ftopos \to \topos$ is called \emph{localic} if the class of objects $\lrset{U \in \ftopos}{U \hookrightarrow f^\ast E \text{ for some } E \in \topos}$ generates the topos $\topos$.
	\end{df}

	\begin{prop}\label{incl-of-grpd-induces-localic}
		Let $\X$ be a logical groupoid.  The geometric morphism $\Sh(i) \colon \Sh(\Y) \to \Sh(\X)$ induced by an inclusion $i \colon \Y \hookrightarrow \X$ of a subgroupoid is localic.
	\end{prop}

	\begin{ex}
		Let $G \subseteq \Omega(X)$ be a subgroup of the topological permutation group on a set $X$ and let $i \colon \B G \to \B \Omega(X)$ be the geometric morphism whose inverse image sends an $\Omega(X)$-set to the same set with the restricted $G$-action.  \cref{incl-of-grpd-induces-localic} asserts that this geometric morphism is localic, however it is instructive to observe this fact without the machinery of \cref{incl-of-grpd-induces-localic}.
		
		It suffices to show that the generators of $\B G$ are subobjects of objects living in the essential image of $i^\ast$.  The generators of $\B G$ can be taken to be the set of cosets $G / K$ (with the obvious action) for $K$ a basic open subgroup of $G$ (see \cite[\S III.9]{SGL} or \cref{ex:moerdijk_generators}), i.e.\ $K = G \cap \mathrm{Stab}(\vec{x})$ where $\mathrm{Stab}(\vec{x}) \subseteq \Omega(X)$ is the stabiliser of some finite tuple $\vec{x} \subseteq X$, say, of length $n$.  We then have that
		\[
		G/K \cong \lrset{g \cdot \vec{x}}{g \in G}  \subseteq X^n,
		\]
		where the isomorphism and inclusion are both equivariant under the obvious $G$-actions.  The isomorphism $G / K \cong \lrset{g \cdot \vec{x}}{g \in G}$ is witnessed by sending $g \cdot \vec{x}$ to $g K$.  Since $X^n$ is evidently in the image of $i^\ast$, we conclude that $i \colon \B G \to \B \Omega(X)$ is localic.  In a similarly direct manner, for a pair of inclusions of topological subgroups $H \subseteq G \subseteq \Omega(X)$, we can show that the geometric morphism induced by restricting actions $\B H \to \B G$ is localic.
		
		However, the assumption that $G, H$ are both topological subgroups of a permutation group $\Omega(X)$, i.e.\ logical groups, is necessary.  Consider the group $(\Q,+)$ of rationals under addition with the Euclidean topology, and the subgroup $(\Z,+) \subseteq (\Q,+)$.  As $\Q$ has no non-trivial open subgroups, there are no non-trivial continuous actions by $\Q$ on discrete sets.  In other words, $\B \Q \simeq \sets$.  But it is easily observed that $\B \Z$ is not localic over $\sets$.
	\end{ex}
	
	\begin{proof}[Proof of \cref{incl-of-grpd-induces-localic}]
		We provide a logical proof.  Let $\theory$ be a geometric theory over a signature $\Sigma$ classified by $\Sh(\X)$.  By \cref{class-repr-grpds-for-theory}, $\X$ can be identified with a groupoid of set-based models for $\theory$, indexed by a set of parameters $\Index$, such that:
		\begin{enumerate}
			\item $X_0$ is a conservative set of models for $\theory$,
			\item the groupoid $\X$ eliminates parameters for the $\Index$-indexing.
		\end{enumerate}
		As $\X$ is a logical groupoid, the topologies on $\X$ are the induced logical topologies.
		
		As a topological subgroupoid $\Y \subseteq \X$, the groupoid $\Y$ can also be identified with a groupoid of $\theory$-models indexed by $\Index$.  It is now relatively easy, using the techniques of \cite[\S 8.5]{myselfgrpds}, to identify a geometric theory represented by the subgroupoid $\Y \subseteq \X$.  For each tuple of parameters $\vec{m} \in \Index$, if the orbit of its interpretation in the groupoid $\Y$, i.e.\ the set
		\[
		\overline{\class{\vec{x} = \vec{m}}}_\Y = \lrset{\lrangle{\vec{n},M}}{\exists N \xrightarrow{\alpha}M \in Y_1 \text{ where } \alpha^{-1}(\vec{n}) = \vec{m}},
		\]
		is not already definable without parameters, i.e.\ of the form $\classv{\phi}{x}_\Y$, we add a relation symbol $R_{\vec{m}}$ to $\Sigma$ whose interpretation in the model $M$ is the subset 
		\[R_{\vec{m}}^M = \lrset{\vec{n} \in M}{\lrangle{\vec{n},M} \in \overline{\class{\vec{x}=\vec{m}}}_\Y}.\]
		For this interpretation, it is easily verified that $\Y$ is a groupoid of $\Index$-indexed structures for the expanded signature $\Sigma'$.  
		
		We take $\theory'$ to be the geometric theory over $\Sigma'$ of the geometric sequents satisfied in all the constituent $\Sigma'$-structures in $\Y$.  By construction, $Y_0$ is a conservative set of models for $\theory'$.  Moreover, $\Y$ clearly eliminates parameters over the signature $\Sigma$ since $\overline{\class{\vec{x} = \vec{m}}}_\Y = \classv{R_{\vec{m}}}{x}_\Y$.

		Since $\Sigma'$ expands $\Sigma$ without adding any new sorts, and $\theory' \supseteq \theory$ as the $\Sigma$-reduct of each $M \in Y_0$ is a $\theory$-model, we have that $\theory'$ is a relational expansion (called a \emph{localic} expansion in \cite[Definition 7.11]{TST}) of $\theory$.  Therefore, by \cite[Theorem 7.1.3]{TST}, the induced geometric morphism $e^{\theory'}_\theory  \colon \topos_{\theory'} \to \topos_\theory $ is {localic}, where $e^{\theory'}_\theory$ is the geometric morphism that sends a $\theory'$-model/homomorphism to its $\theory$-reduct.  
		
		It remains to show that the geometric morphisms $e^{\theory'}_\theory$ and $\Sh(i)$ are naturally isomorphic.  It suffices to show that inverse images agree on the {syntactic category} $\cat_\theory \subseteq \topos_\theory$.  Recall that ${e^{\theory'}_\theory}^\ast$ sends the object $\form{\phi}{x}_\theory \in \cat_\theory$ to $\form{\phi}{x}_{\theory'} \in \cat_{\theory'} \subseteq \topos_{\theory'}$.  Under the equivalence $\topos_\theory \simeq \Sh(\X)$, $\form{\phi}{x}_\theory$ corresponds to the definable sheaf $\pi_{\classv{\phi}{x}} \colon \classv{\phi}{x}_\X \to X_0$, whose pullback along the inclusion $i \colon \Y \to \X$ is $\pi_{\classv{\phi}{x}} \colon \classv{\phi}{x}_\Y \to Y_0$, which corresponds to $\form{\phi}{x}_{\theory'}$ under the equivalence $\topos_{\theory'} \simeq \Sh(\Y)$.  This witnesses the natural isomorphism $\Sh(i) \cong e^{\theory'}_\theory$, and so $\Sh(i)$ is localic as desired.
	\end{proof}

	\paragraph{When localic geometric morphisms are equivalences.}

	Therefore, recall from \cite[\S A4.6]{elephant} that the induced geometric morphism $\Sh(i) \colon \Sh(\Y) \to \Sh(\X)$ is an equivalence if and only if it is \emph{hyperconnected} as well as localic.  Hence, by \cite[Proposition A4.6.6]{elephant}, we have that:
	
	\begin{coro}\label{coro:weqv_iff_hyp}
		The inclusion $i \colon \Y \hookrightarrow \X$ of a subgroupoid of a logical groupoid is a weak equivalence if and only if, for each object $W \in \Sh(\X)$, or indeed each object $W $ in a generating set of objects for $\Sh(\X)$, the induced map on subobjects
		\[
		\Sh(i)_W^\ast \colon \Sub_{\Sh(\X)}(W) \to \Sub_{\Sh(\Y)}(\Sh(i)^\ast W)
		\]
		is an isomorphism.
	\end{coro}
	
	Different choices of generating set will yield different, but equivalent conditions on when $i$ is a weak equivalence.
	
	\begin{ex}\label{ex:Morita-equivalent-subgroupoids-when-theory}
		Let $\theory$ be a geometric theory and suppose that $\X$ is a groupoid of $\theory$-models with an indexing by a set of parameters $\Index$ such that:
		\begin{enumerate}
			\item $X_0$ is a conservative set of models for $\theory$.
			\item the groupoid $\X$ eliminates the indexing parameters,
		\end{enumerate}
		and the topologies on $\X$ are the logical topologies for this indexing.  Thus, by \cref{class-repr-grpds-for-theory}, the topos $\Sh(\X)$ classifies $\theory$.  If $i \colon \Y \hookrightarrow \X$ is a topological subgroupoid, $\Y$ is also a groupoid of $\theory$-models equipped with an indexing by $\Index$.  For a definable sheaf $\classv{\phi}{x}_\X$ from \cref{ex:definable_sheaves}, the pullback $\Sh(i)^\ast \classv{\phi}{x}_\X$ is given by $\classv{\phi}{x}_\Y$.  Moreover, the induced map on subobjects
		\[
		\Sh(i)_{\classv{\phi}{x}_\X}^\ast \colon \Sub_{\Sh(\X)}(\classv{\phi}{x}_\X) \to \Sub_{\Sh(\Y)}(\classv{\phi}{x}_\Y)
		\]
		is an isomorphism, for each formula $\phi$, if and only if $\Y$ is also a conservative groupoid of models for $\theory$ and eliminates parameters for the given indexing, i.e.\ $\Sh(\Y)$ also classifies the theory $\theory$ (as we would expect if there were an equivalence $\Sh(\Y) \simeq \Sh(\X)$).
	\end{ex}
	
	\subsection{The Moerdijk generators}
	
	In \cref{ex:Morita-equivalent-subgroupoids-when-theory}, we have, in effect, chosen a generating set of objects for $\Sh(\X)$ -- the definable sheaves.  We instead desire an entirely topological description of the weak equivalences $\weqv$, in other words an entirely topologically defined generating set of objects for the topos $\Sh(\X)$, devoid of any mention of a geometric theory that it classifies.  Just such a generating set was identified by Moerdijk in \cite[\S 6.1]{cont1}.

	\begin{df} Let $\X$ be an open topological groupoid.  Given an open subgroupoid $\U \subseteq \X$, we obtain an $\X$-sheaf as follows.
			\begin{enumerate}
				\item The underlying local homeomorphism is the continuous map
				\[
				\bar{t} \colon {[s^{-1}(U_0)]_\U} \to X_0,
				\]
				the map that sends the $U_1$-orbit (by pre-composition) $\left[x \xrightarrow{\alpha} y\right]_\U$, where $x \in U_0$, to $y \in X_0$.
				\item The $X_1$-action on $[s^{-1}(U_0)]_\U$ is given by (post-)composition, i.e.\ $\beta \cdot [\alpha] = [\beta \circ \alpha]$.
			\end{enumerate}
			We denote this object by $\lrangle{\U}$ (details that $\bar{t}$ is a local homeomorphism, or that $[- \circ -]$ defines an $X_1$-action, can be found in \cite{cont1}).
	\end{df}
	\begin{prop}[\S 6.1 \cite{cont1}]
		For an open topological groupoid $\X$, the objects of the form $\lrangle{\U}$, with $\U \subseteq \X$ an open subgroupoid, yield a generating set for the topos $\Sh(\X)$.
	\end{prop}
	\begin{ex}\label{ex:moerdijk_generators}
			

			 Let $G$ be a topological group.  The open subgroupoids of $G$ are the open subgroups $H \subseteq G$, and $\lrangle{H}$ is the set of $H$-cosets $ G/H$ with the obvious (and continuous) $G$-action.  Thus, the Moerdijk generators coincide with the generators for $\B G$ identified in \cite[\S III.9]{SGL}.
	\end{ex}
	
	\begin{ex}[cf.\ Lemma 2.1.5 \& Corollary 2.1.6 \cite{awodeyforssell}]
		We now relate the Moerdijk generators with definable sheaves from \cref{ex:definable_sheaves}.  Let $\X$ be a groupoid of models for a geometric theory $\theory$.  Suppose that $\X$ admits an indexing by a parameter set $\Index$ such that, for a tuple of parameters $\vec{m} \in \Index$,
		\[
		\overline{\class{\vec{x}=\vec{m}}}_\X = \classv{\chi_{\vec{m}}}{x}_\X,
		\]
		i.e.\ the formula $\chi_{\vec{m}}$ eliminates the parameters $\vec{m}$ for the groupoid $\X$.  We claim that there is a homeomorphism of sheaves
		\[
		\lrangle{
		\X^{\vec{m}}
		}
		\cong \overline{\class{\vec{x}=\vec{m}}}_\X = \classv{\chi_{\vec{m}}}{x}_\X,
		\]
		where $\X^{\vec{m}} \subseteq \X$ is the open subgroupoid of all models $M \in X_0$ that interpret the tuple $\vec{m}$ and the isomorphisms $M \xrightarrow{\alpha} M' \in X_1$ that preserve the interpretation of these parameters, that is to say, the open subgroupoid corresponding to the open subset
		\[
		\lrclass{
			\begin{matrix}
				\vec{m} : \top \\
				\vec{m} \mapsto \vec{m} \\
				\vec{m} : \top
			\end{matrix}
		}_\X \subseteq X_1.
		\]
		
		To construct the map $\lrangle{\X^{\vec{m}}} \to \overline{\class{\vec{x} = \vec{m}}}_\X$ witnessing the isomorphism, given an arrow $M \xrightarrow{\alpha} N$ with $M \in X_0^{\vec{m}}$, i.e.\ the model $M$ interprets the parameters $\vec{m}$, we send $[\alpha]_{\X^{\vec{m}}} \in \lrangle{\X^{\vec{m}}}$ to the pair $\lrangle{\alpha(\vec{m}), N} \in  \overline{\class{\vec{x} = \vec{m}}}_\X$.  This is a well-defined map since, if $[\alpha]_{\X^{\vec{m}}} = [\alpha']_{\X^{\vec{m}}}$, then there is an arrow $M \xrightarrow{\gamma} M' \in X_1^{\vec{m}}$ yielding a commuting diagram
		\[
		\begin{tikzcd}[row sep=tiny]
			M \ar{dd}[']{\gamma} \ar{rd}{\alpha} & \\
			& N ,\\
			M' \ar{ru}[']{\alpha'}
		\end{tikzcd}
		\]
		and so $\alpha(\vec{m}) = \alpha'(\vec{m})$ as $\gamma$ preserves the interpretation of the parameters $\vec{m}$.  
		
		In the converse direction, we simply send $\lrangle{\vec{n},N} \in \overline{\class{\vec{x} = \vec{m}}}_\X$ to the orbit of the isomorphism $M \xrightarrow{\alpha} N \in X_1$ for which $M$ interprets $\vec{m}$ and $\alpha(\vec{m}) = \vec{n}$.  Again, this is a well defined map since given a second isomorphism $M' \xrightarrow{\alpha'} N \in X_1$ for which $\vec{n} = \alpha'(\vec{m})$, then $M' \xrightarrow{\alpha^{-1} \circ \alpha'} M$ is an isomorphism of models that preserves the interpretation of the tuple of parameters $\vec{m}$, and clearly $\alpha' = \alpha \circ \alpha^{-1} \circ \alpha'$, i.e.\ $[\alpha]_{\X^{\vec{m}}} = [\alpha']_{\X^{\vec{m}}}$.  We omit the additional details that these maps constitute a homeomorphism of sheaves.  In an identical fashion, for any geometric formula $\phi$, we can demonstrate a homeomorphism of sheaves $\classv{\chi_{\vec{m}} \land \phi}{x}_\X \cong \lrangle{\X^{\phi(\vec{m})}}$, where $\X^{\phi(\vec{m})} \subseteq \X$ is the open subgroupoid corresponding to the open subset
		\[
		\lrclass{
			\begin{matrix}
				\vec{m} : \phi \\
				\vec{m} \mapsto \vec{m} \\
				\vec{m} : \phi
			\end{matrix}
		}_\X \subseteq X_1.
		\]

		Suppose that $\X$ is conservative and eliminates all parameters.  Then $\Sh(\X)$ classifies $\theory$ by \cref{class-repr-grpds-for-theory}.  Recall that, in this case, the definable sheaves $\classv{\phi}{x}_\X$ generate the topos $\Sh(\X)$, and observe as well that $\classv{\phi}{x}_\X = \bigcup_{\vec{m} \in \Index} \classv{\chi_{\vec{m}} \land \phi}{x}_\X$ (since every tuple in each model $M \in X_0$ must be the interpretation of some tuple of parameters).  Note also that the open subgroupoids of the form $\X^{\phi(\vec{m})}$ yield a \emph{basis of open subgroupoids for} $\X$ (in the terminology of Moerdijk \cite[\S 6.5]{cont1}).  Therefore, the objects
		\[
		\lrset{\classv{\chi_{\vec{m}} \land \phi }{x}_\X}{ \vec{m} \in \Index, \, \phi \text{ a formula}}
		\]
		are a generating set of objects for $\Sh(\X)$ that are contained in the intersection of the definable sheaves and the Moerdijk generators.
	\end{ex}

	\paragraph{Restricting to a subgroupoid.}
	
	We now use the Moerdijk generators to prove \cref{classification-of-weqv}.  Let $i \colon \Y \hookrightarrow \X$ be a subgroupoid inclusion.  The image $\Sh(i)^\ast \lrangle{\U}$ under the inverse image functor, i.e.\ the pullback along $i$, has $[s^{-1}(U_0) \cap t^{-1}(Y_0) ]_\U$ as its underlying space.  The $Y_1$-action is simply the restriction of the $X_1$-action.
	
	\begin{rem}
		Note that $i^{-1} \U = (U_1 \cap Y_1 \rightrightarrows U_0 \cap Y_0)$ defines an open subgroupoid of $\Y$.  However, note also that, in general, $\Sh(i)^\ast \lrangle{ \U} \not \cong \lrangle{i^{-1} \U}$.  Conditions for when there is an isomorphism between $\Sh(i)^\ast \lrangle{ \U}  $ and $ \lrangle{i^{-1} \U}$ are discussed in \cite[Lemma 2.1.2]{awodeyforssell}.
	\end{rem}
	
	\begin{proof}[Proof of \cref{classification-of-weqv}]
		By \cref{coro:weqv_iff_hyp}, the geometric morphism $\Sh(i)$ induced by an inclusion of a subgroupoid $i \colon \Y \hookrightarrow \X$ is an equivalence of topoi $\Sh(\Y) \simeq \Sh(\X)$ if and only if the induced map on subobjects is an isomorphism for some set of generators of $\Sh(\X)$.  In particular, taking the Moerdijk generators, we have that $\Sh(i)$ is an equivalence if and only if, for each open subgroupoid $\U \subseteq \X$, the map
		\[
		\Sh(i)_{\lrangle{\U}}^\ast \colon \Sub_{\Sh(\X)}(\lrangle{\U}) \to \Sub_{\Sh(\Y)}(\Sh(i)^\ast \lrangle{\U})
		\]
		is an isomorphism.  Recall that a subobject of $\lrangle{\U} \in \Sh(\X)$ is an open subset of $[s^{-1}(U_0)]_\U$ that is stable under the action of $X_1$ by post-composition, and therefore can be identified with an open subset of ${ _\X[s^{-1}(U_0)]_\U}$.  Similarly, each subobject of $\Sh(i)^\ast \lrangle{\U} \in \Sh(\Y)$ corresponds to an open subset of ${ _\Y[s^{-1}(U_0) \cap t^{-1}(Y_0)]_\U}$.  The map on subobjects $\Sh(i)_{\lrangle{\U}}^\ast$ sends a subobject $W \subseteq [s^{-1}(U_0)]_\U$ to $W \cap [t^{-1}(Y_0)]_\U \subseteq [s^{-1}(U_0) \cap t^{-1}(Y_0)]_\U$.  Composing with the isomorphisms
		\[
		\opens({ _\X[s^{-1}(U_0)]_\U}) \cong  \Sub_{\Sh(\X)}(\lrangle{\U}) \xrightarrow{\Sh(i)_{\lrangle{\U}}^\ast } \Sub_{\Sh(\Y)}(\Sh(i)^\ast \lrangle{\U}) \cong \opens({ _\Y[s^{-1}(U_0) \cap t^{-1}(Y_0)]}_\U),
		\]
		the resultant map is precisely the inverse image map $\iota^{-1} \colon \opens({ _\X[s^{-1}(U_0)]_\U}) \to \opens({ _\Y[s^{-1}(U_0) \cap t^{-1}(Y_0)]}_\U)$.  Thus, $\Sh(i)^\ast $ is an isomorphism if and only if $\iota^{-1}$ is bijective, i.e.\ if 
		\[\iota \colon { _\Y[s^{-1}(U_0) \cap t^{-1}(Y_0)]}_\U \to { _\X[s^{-1}(U_0)]_\U}\]
		is a quasi-homeomorphism.
	\end{proof}

		\paragraph{Logical groups.}
	
	We now demonstrate that, in the case of logical groups, the conditions of \cref{classification-of-weqv} simplify and align with our expectations from \cite[\S 4]{hodges}.
	
	\begin{prop}\label{groups-thick-becomes-dense}
		Let $G$ be a logical group.  The weak equivalences of topological groups with codomain $G$ are precisely the inclusions of the dense subgroups of $G$.
	\end{prop}
	\begin{proof}
		Let $H \subseteq G$ be a subgroup.  For $g \in G$, recall that the subsets $g K$, where $K$ is an open subgroup of $G$, define a basis of open neighbourhoods of $g$.  Therefore, $H$ is dense in $G$ if and only if, for all $g \in G$ and open subgroups $K \subseteq G$, we have that $g K \cap H \neq \emptyset$.
		
		Observe that, for any $g \in G$, the bi-orbit ${ _G[g]_K}$ is the whole space $G$, and thus ${ _G[G]_K}$ is the one-point space.  We also note that the space ${ _H[G]_K}$ is discrete -- this is because ${ _H[gK]_K} $ is an open neighbourhood of ${ _H[g]_K} \in { _H[G]_K}$, but it is easily calculated that ${ _H[gK]_K}$ is the singleton $\{\, { _H[g]_K} \,\}$.  Thus, the map
		\[
		\iota^{-1} \colon { _H[G]_K} \to { _G[G]_K}  \cong \{\ast\}
		\]
		is a quasi-homeomorphism if and only if ${ _H[g]_K} = { _H[g']_K} $ for all $g, g' \in G$, i.e.\ the double co-sets $HgK$ and $Hg'K$ are equal, or equivalently $g' \in HgK$.
		
		Suppose that $H \subseteq G$ is dense.  Then, for any $g, g' \in G$, since $H \cap gK, H \cap g'K \neq \emptyset$, there exists $h,h' \in H$ and $k,k' \in K$ such that $h = gk$ and $h' = g'k'$.  A simple calculation then yields that $g' = h'h^{-1}gk{k'}^{-1}$, i.e.\ $g' \in HgK$, and so $\iota^{-1} \colon { _H[G]_K} \to { _G[G]_K}  $ is a quasi-homeomorphism.  Now suppose that $\iota^{-1}$ is a quasi-homeomorphism.  For each $g \in G$, we have that the identity $e$ is contained in $HgK$, and so there exist $h \in H$ and $k \in K$ for which $h = gk$, i.e.\ $H \cap gK \neq \emptyset$.  Thus, $H \subseteq G$ is dense.
	\end{proof}

	\subsection{Localic surjections and subtopos inclusions}

	We have demonstrated that, given a subgroupoid $i \colon \Y \hookrightarrow \X$ of a logical groupoid, $\Sh(i)$ is an equivalence if and only if $\iota^{-1} \colon \opens({ _\X[s^{-1}(U_0)]_\U}) \to \opens({ _\Y[s^{-1}(U_0) \cap t^{-1}(Y_0)]}_\U)$ is bijective, for each open subgroupoid $\U \subseteq \X$.  It is of interest to split up this condition into when $\iota^{-1}$ is injective and surjective.  By the results of \cite{myselfintloc}, the former occurs to when $\Sh(i)$ is a \emph{localic surjection}, while the latter corresponds to $\Sh(i)$ being a \emph{subtopos inclusion}.
	
	\paragraph{Skula dense orbits.}  We begin by studying when $\Sh(i)$ is a surjection, i.e.\ when the inverse image part $\Sh(i)^\ast$ is faithful.  We first note that, for each open subgroupoid $\U \subseteq \X$, the space ${ _\X[s^{-1}(U_0)]_\U}$ is homeomorphic to the space of $\U$-\emph{orbits}, that is the quotient space
	\[
	\begin{tikzcd}
		U_1 \ar[shift left]{r}{s} \ar[shift right]{r}[']{t} & U_0 \ar[two heads]{r} & \pi(\U).
	\end{tikzcd}
	\]
	The homeomorphism is witnessed by sending the $\U$-orbit of an object $x \in U_0$ to ${ _\X[\id_x]_\U} \in { _\X[s^{-1}(U_0)]_\U}$ and, for the inverse, ${ _\X[\alpha]_\U} \in { _\X[s^{-1}(U_0)]_\U} $ to the $\U$-orbit of $s(\alpha)$ (cf.\ \cite[Lemma 2.1.2]{awodeyforssell}).  The image of the continuous map
	\[
	\iota \colon { _\Y[s^{-1}(U_0) \cap t^{-1}(Y_0)]_\U} \to { _\X[s^{-1}(U_0)]_\Y} \cong \pi(\U)
	\]
	can easily be described -- it is the subspace the subspace of $\pi(\U)$ of $\U$-orbits that intersect $\overline{Y_0}$ (i.e.\ the set of all objects that are isomorphic, in $\X$, to an object in $Y_0 \subseteq X_0$).  We denote this subspace by $\pi(i^\ast \U) \subseteq \pi(\U)$.
	
	\begin{df}\label{df:skula-dense}
		Let $\X$ be a logical groupoid and $\U \subseteq \X$ an open subgroupoid.  An inclusion $i \colon \Y \hookrightarrow \X$ of a second subgroupoid will be said to have \emph{Skula dense $\U$-orbits} if, given a pair of open subsets $W, W' \subseteq U_0$, whenever $W' \cap \overline{Y_0} \subseteq W \cap \overline{Y_0}$, then for every $x \in W'$ there is some arrow $x \xrightarrow{\alpha} z \in U_1$ with $z \in W$.  Our terminology is justified by the following lemma.
	\end{df}
	
	\begin{fact}
		Recall that the \emph{Skula topology} on a space $X$ is the topology generated by both the open and closed subsets of $X$.  A subset $A \subseteq X$ is said to be \emph{Skula dense} if it is dense for the Skula topology.  Recall from \cite[Proposition IV.10.1.2]{EGAIV} that a subspace $A \subseteq X$ is Skula dense if and only if the inclusion $A \subseteq X$ is a quasi-homeomorphism.
	\end{fact}
	
	\begin{lem}
		Let $i \colon \Y \hookrightarrow \X$ be the inclusion of a subgroupoid and let $\U \subseteq \X$ be an open subgroupoid.  Then \cref{df:skula-dense} is satisfied if and only if $\pi(i^\ast \U) $ is a Skula dense subspace of $\pi(\U)$.
	\end{lem}
	\begin{proof}
		Note that an open of $\pi(\U)$ is described by an open $W \subseteq U_0$ that is stable under the action of $U_1$, i.e.\ if $x \in W$ and $x \xrightarrow{\alpha} z \in U_1$, then $z \in W$ too.  The induced map on opens $\opens(\pi(\U)) \to \opens(\pi(i^\ast \U))$ acts on these opens by sending $W$ to $W \cap \overline{Y_0}$.  Note that this map is automatically surjective.  It is injective if and only if whenever $W' \cap \overline{Y_0} \subseteq W \cap \overline{Y_0}$, for some other open $W' \subseteq U_0$, then the $U_1$-orbit of $W'$ is contained in $W$.  The equivalence with condition \cref{df:skula-dense} is now easily recognised.
	\end{proof}
	
	\begin{prop}\label{surjection-iff-skula-dense-orbits}
		Let $i \colon \Y \hookrightarrow \X$ be the inclusion of subgroupoid into a logical groupoid.  The induced geometric morphism $\Sh(i) \colon \Sh(\Y) \to \Sh(\X)$ is a localic surjection if and only if $\Y$ has Skula dense $\U$-orbits for every open subgroupoid $\U \subseteq \X$.
	\end{prop}
	\begin{proof}
		By \cite{myselfintloc}, $\Sh(i)$ is a geometric surjection if and only if, for each open subgroupoid $\U \subseteq \X$, the induced map on subobjects $\Sh(i)^\ast_{\lrangle{\U}} \colon \Sub_{\Sh(\X)}(\lrangle{\U}) \to \Sub_{\Sh(\Y)}(\Sh(i)^\ast \lrangle{\U})$ is injective.  Recall from the proof of \cref{classification-of-weqv} that $\Sh(i)^\ast_{\lrangle{\U}}$ is isomorphic to the inverse image map
		\[
		\iota^{-1} \colon \opens({ _\X[s^{-1}(U_0)]_\U}) \to \opens({ _\Y[s^{-1}(U_0) \cap t^{-1}(Y_0)]}_\U).
		\]
		Recall also, from \cite[\S IV.10.2]{EGAIV}, that $\iota^{-1}$ is injective if and only if the inclusion of the image of $\iota$, that is the subspace inclusion $\pi(i^\ast \U) \subseteq \pi(\U)$, is a quasi-homeomorphism, thus completing the result.
	\end{proof}

	\begin{ex}\label{ex:groups-always-skula-dense}
		Let $H \subseteq G \subseteq \Omega(X)$ be the inclusions of topological groups.  An open subgroupoid of $G$, i.e.\ an open subgroup, only ever has one object, and so trivially the restricted orbit spaces are Skula dense.  Thus, by \cref{surjection-iff-skula-dense-orbits}, the induced geometric morphism $\B H \to \B G$ is surjective.  In fact, if we only desired $\B H \to \B G$ to be a surjection and omitted the requirement to be localic, then we do not need to assume that $H, G$ are subgroups of $\Omega(X)$.  For any subgroup of a topological group, the restriction of action functor is evidently faithful, and so $\B H \to \B G$ is surjective (cf.\ \cref{surj-on-objs-implies-surj}).
	\end{ex}
	
	\begin{ex}\label{ex:quasihomeo-for-spaces}
		Let $X$ be a topological space and $Y \subseteq X$ a subspace.  Both $X$ and $Y$ can be viewed as categorically discrete topological groupoids.  In this case, the condition that the inclusion $Y \subseteq X$ is Skula dense on orbits (of open subgroupoids) trivialises to requiring that $Y \subseteq X$ is Skula dense, i.e.\ $\opens(Y) \cong \opens(X)$, which is necessary and sufficient for there to be an equivalence of sheaf topoi $\Sh(Y) \simeq \Sh(X)$ (indeed, the induced geometric morphism $\Sh(Y) \to \Sh(X)$ is already the inclusion of a subtopos, see \cref{ex:inclusions-of-disc-grpds-are-thick}).
	\end{ex}
	
	\begin{prop}\label{skula-dense-orbits-is-wide}
		The class of subgroupoid inclusions $i \colon \Y \hookrightarrow \X$ which have Skula dense $\U$-orbits, for all open subgroupoids $\U \subseteq \X$, is a wide class.
	\end{prop}
	\begin{proof}
		As in \cref{weqv-is-wide}, the statement follows abstractly from \cref{surjection-iff-skula-dense-orbits} since localic surjections are closed under composition and natural isomorphisms.  For clarity, we include an explicit argument that the class is closed under continuous isomorphisms.  
		
		Let $i,j \colon \Y \hookrightarrow \X$ be continuously isomorphic inclusions of $\Y$ as a subgroupoid of $\X$, with $a \colon i \Rightarrow j$ as the isomorphism, and let $\U \subseteq \X$ be an open subgroupoid.  We wish to show that if $j$ has Skula dense $\U$-orbits, then so too does $i$.  But since \cref{df:skula-dense} only concerns the orbits of the images of the inclusions, which are identical as the images $i(Y_0)$ and $j(Y_0)$ are naturally isomorphic, the claim follows immediately. 
	\end{proof}
	
	\paragraph{Source determined orbits of opens.}  We now turn to study when $\Sh(i)$ is the inclusion of a subtopos, i.e.\ when the direct image functor $\Sh(i)_\ast$ is fully faithful.

	\begin{df}
		Let $\X$ be a logical groupoid, $i \colon \Y \hookrightarrow \X$ a subgroupoid, and $\U \subseteq \X$ an open subgroupoid.  We will say that an open
				\[
				V \subseteq s^{-1}(U_0) \cap t^{-1}(Y_0) = \lrset{x \xrightarrow{\alpha} y}{x \in U_0 \text{ and } y \in Y_0 }
				\]
		has a \emph{source determined orbit} if, for each arrow $x \xrightarrow{\alpha} y \in V$, there is an open neighbourhood $ W_\alpha \subseteq U_0$ of $x$ such that, for any other morphism $x' \xrightarrow{\gamma} y'$ with $x' \in W_\alpha$ and $y' \in Y_0$, there is a commuting square
				\[
				\begin{tikzcd}
						x' \ar{r}{\gamma} \ar{d}[']{\zeta} & y' \\
						x'' \ar{r}{\eta} & y'' \ar{u}[']{\theta}
					\end{tikzcd}
				\]
		where $\zeta \in U_1$, $\eta \in V$ and $\theta \in Y_1$.  This is equivalent to saying that there exists an open subset $W \subseteq U_0$ for which ${ _\Y[V]_\U} = s^{-1}(W) \cap t^{-1}(Y_0)$ (such a $W \subseteq U_0$ is given by $\bigcup_{\alpha \in V} W_\alpha$).	
	\end{df}
	
	\begin{prop}\label{subtopos-iff-thick}
		Let $i \colon \Y \hookrightarrow \X$ be the inclusion of a subgroupoid in a logical groupoid.  Then the geometric morphism $\Sh(i) \colon \Sh(\Y) \to \Sh(\X)$ is the inclusion of a subtopos if and only if, for every open subgroupoid $\U \subseteq \X$, every open $V \subseteq s^{-1}(U_0) \cap t^{-1}(Y_0)$ has a source determined orbit.
	\end{prop}
	\begin{proof}
		Recall from \cite{myselfintloc} that $\Sh(i)$ is a geometric inclusion if and only if the subobject map 
		\[\Sh(i)^\ast_{\lrangle{\U}} \colon \Sh_{\Sh(\X)}(\lrangle{\U}) \to \Sub_{\Sh(\Y)}(\Sh(i)^\ast \lrangle{\U})\]
		is surjective for each open subgroupoid $\U \subseteq \X$.  Recall also (from \cite[Lemma 2.1.2]{awodeyforssell}) that the subobjects of $\lrangle{\U}$ correspond to $U_1$-stable open subsets of $U_0$ and that the subobjects of $\Sh(i)^\ast \lrangle{\U}$ correspond to opens of ${ _\Y[s^{-1}(U_0) \cap t^{-1}(Y_0)]_\U}$, i.e.\ the sets ${ _\Y[V]_\U}$ for an open subset $V \subseteq s^{-1}(U_0) \cap t^{-1}(Y_0)$.  Thus, $\Sh(i)$ is an inclusion if and only if, for each open $V \subseteq s^{-1}(U_0) \cap t^{-1}(Y_0)$, there exists a $U_1$-stable open $W \subseteq U_0$ such that ${ _\Y[V]_\U} = s^{-1}(W) \cap t^{-1}(Y_0)$.  The result now follows by observing that, if ${ _\Y[V]_\U} = s^{-1}(W') \cap t^{-1}(Y_0)$ for some subset $W' \subseteq U_0$, then $W'$ must be $U_1$-stable.
	\end{proof}
	
	\begin{ex}\label{ex:inclusions-of-disc-grpds-are-thick}
		Let $Y \subseteq X$ be a subspace inclusion viewed as the inclusion of categorically discrete topological groupoids as in \cref{ex:quasihomeo-for-spaces}.  Such an inclusion automatically has source determined stabilisations.  This is because, being categorically discrete, for any open subset $U \subseteq X$ (i.e.\ an open subgroupoid), an open $V \subseteq s^{-1}(U_0) \cap t^{-1}(Y_0) = \lrset{x = y}{x \in U_0, y \in Y_0}$ is just an open $V \subseteq U_0 \cap Y_0$, and hence evidently source determined (for $x = y \in V$, take any open neighbourhood $W \subseteq U_0$ of $x$).  Thus, by applying \cref{subtopos-iff-thick} we return the familiar fact that $\Sh(Y)$ is a subtopos of $\Sh(X)$.
	\end{ex}
	
	\begin{prop}\label{thick-is-wide}
		The class of subgroupoid inclusions $i \colon \Y \hookrightarrow \X$ for which every open $V \subseteq s^{-1}(U_0) \cap t^{-1}(Y_0)$ has determined orbit, for each open subgroupoid $\U \subseteq \X$, is a wide class.
	\end{prop}
	\begin{proof}
		As in \cref{weqv-is-wide} and \cref{skula-dense-orbits-is-wide}, this follows from \cref{subtopos-iff-thick} since subtopos inclusions are closed under composition and natural isomorphisms.  We include a direct proof that the class is closed under continuous isomorphisms.
		
		Let $i, j \colon \Y \hookrightarrow \X$ be continuously isomorphic inclusions of $\Y$ as a subgroupoid of $\X$, with $a \colon i \Rightarrow j$ as the isomorphism, and let $\U \subseteq \X$ be an open subgroupoid.  Suppose that each open subset of $s^{-1}(U_0) \cap t^{-1}(j_0(Y_0))$ has a source determined open.  We wish to show the same for $i$.  Let
			\[
			V \subseteq \lrset{x \xrightarrow{\alpha} i_0(y)}{x \in U_0 \text{ and } y \in Y_0} = s^{-1}(Y_0) \cap t^{-1}(i_0(Y_0))
			\]
			be an open subset.  We need to show that, for each $x \xrightarrow{\alpha} i_0(y) \in V$, there is an open neighbourhood $W \subseteq U_0$ of $x$ such that any other map $x' \xrightarrow{\gamma} i_0(y')$ in $s^{-1}(W) \cap t^{-1}(i_0(Y_0))$ can be rewritten as a composite $i_1(\theta) \circ \eta \circ \zeta$ with $\zeta \in U_1$, $\eta \in V$ and $\theta \in Y_1$.  Let $V'$ denote the subset
			\[
			V' = \lrset{x \xrightarrow{\alpha} i_0(y) \xrightarrow{a_y} j_0(y)}{\alpha \in V} \subseteq \lrset{x \xrightarrow{\alpha} j_0(y)}{x \in X_0, y \in Y_0}.
			\]
			Since $a$ is a continuous isomorphism, $V'$ is an open subset of $s^{-1}(Y_0) \cap t^{-1}(j_0(Y_0))$.  By assumption, there exists an open neighbourhood $W$ of $x$ and arrows $\zeta \in U_1$, $\eta' \in V'$ and $\theta \in Y_1$ such that the square
			\[
			\begin{tikzcd}
					x' \ar{r}{\gamma} \ar{d}[']{\zeta} & i_0(y') \ar{r}{a_{y'}} & j_0(y') \\
					x'' \ar{rr}{\eta'} && j_0(y'') \ar{u}[']{j_1(\theta)}
				\end{tikzcd}
			\]
			commutes.  As $\eta' \in V'$, it is of the form $x'' \xrightarrow{\eta} i_0(y'') \xrightarrow{a_{y''}} j_0(y'')$, and so using the naturality of $a$, we have a commuting diagram
			\[
			\begin{tikzcd}
					x' \ar{r}{\gamma} \ar{d}[']{\zeta} & i_0(y') \ar{r}{a_{y'}} & j_0(y') \\
					x'' \ar{r}{\eta} & i_0(y'') \ar{u}{i_1(\theta)} \ar{r}{a_{y''}} &  j_0(y'') \ar{u}[']{j_1(\theta)},
				\end{tikzcd}
			\]
		the left-hand square of which is the desired commuting square, completing the proof.
	\end{proof}
	
	Since a geometric morphism is an equivalence if and only if it is both a surjection and an inclusion (\cite[Corollary A4.2.11]{elephant}), we obtain the following alternative description of the class of weak equivalences:
	
	\begin{coro}\label{weqv-skula-dense-orbits-and-thick}
		Let $\X$ be a logical groupoid and $i \colon \Y \hookrightarrow \X$ a subgroupoid.  The induced geometric morphism $\Sh(i)$ is an equivalence of topoi if and only if, for each open subgroupoid $\U \subseteq \X$,
		\begin{enumerate}
			\item $\Y$ has Skula dense $\U$-orbits,
			\item every open $V \subseteq s^{-1}(U_0) \cap t^{-1}(Y_0)$ has a source determined orbit.
		\end{enumerate}
	\end{coro}

	\paragraph{Replete subgroupoids.}  In \cite{forssell-subgroupoid}, Forssell shows that the inclusion of a full and \emph{replete} topological subgroupoid induces a subtopos of the topos of sheaves.  We recover this result in the special case of logical groupoids by demonstrating that a full, replete subgroupoid automatically satisfies \cref{subtopos-iff-thick}.
	
	\begin{df}
		A subgroupoid $\Y \hookrightarrow \X$ is \emph{replete} if it is closed under isomorphisms (i.e.\ if $x \xrightarrow{\alpha} y \in X_1$ is an arrow with $x \in Y_0$, then $y \in Y_0$ too).  Thus, a replete, \emph{full} subgroupoid $\Y \hookrightarrow \X$ (i.e.\ a replete subgroupoid for which $Y_1 = s^{-1}(Y_0) \cap t^{-1}(Y_0)$) is determined by any subset $Y_0 \subseteq X_0$ that is $X_1$-stable (i.e.\ $Y_0 = ts^{-1}(Y_0)$).
	\end{df}

	\begin{coro}[Theorem 2.11 \cite{forssell-subgroupoid}]\label{full-replete-yields-subtopos}
		Let $\X$ be a logical groupoid.  If $i \colon \Y \hookrightarrow \X$ the inclusion of a full, replete subgroupoid, then $\Sh(i)$ the inclusion of a subtopos.
	\end{coro}
	\begin{proof}
		We will show that the inclusion of full, replete subgroupoid satisfies the hypotheses of \cref{subtopos-iff-thick}.  Let $V \subseteq s^{-1}(U_0) \cap t^{-1}(Y_0)$ be an open subset and let $x \xrightarrow{\alpha} y \in V$.  As $s$ is an open map, the subset $ s(V) \subseteq U_0 \cap s(t^{-1}(Y_0)) = U_0 \cap Y_0$ is an open neighbourhood of $x$.  Therefore, by the definition of the subspace topology, there is an open subset $W \subseteq U_0$ such that $x \in W \cap Y_0 \subseteq s(V)$.
		
		We claim that $s^{-1}(W) \cap t^{-1}(Y_0) \subseteq { _\Y[V]_\U}$.  Given any arrow $x' \xrightarrow{\gamma} y'$ with $x' \in W$ and $y' \in Y_0$, since $\Y \subseteq \X$ is replete, this entails that $x' \in W \cap Y_0$, and hence $x' \in s(V)$.  Thus, there is some arrow $x' \xrightarrow{\eta} y'' \in V$, and so a commuting triangle
		\[
		\begin{tikzcd}[row sep=tiny]
			& y' \\
			x' \ar[bend left]{ru}{\gamma} \ar[bend right]{rd}[']{\eta} \\
			& y'' \ar{uu}[']{\gamma \circ \eta^{-1}}.
		\end{tikzcd}
		\]
		Since $\Y \subseteq \X$ is full, $\gamma \circ \eta^{-1} \in Y_1$, thus concluding the proof.
	\end{proof}

	\subsection{The surjection-inclusion factorisation}
	
	Let $\Phi \colon \X \to \Y$ be a continuous functor between logical groupoids.  Recall from \cite[Theorem A4.2.10]{elephant} that the geometric morphism $\Sh(\Phi)$ factorises uniquely (up to isomorphism) as a surjection followed by a subtopos inclusion.  Using the results of the previous subsection, we can witness this factorisation at the level of topological groupoids -- it is induced by the \emph{full essential image factorisation} of $\Phi$.
	
	\begin{df}
		For a continuous functor $\Phi \colon \X \to \Y$ between topological groupoids, let $\Phi(\X)$ denote the \emph{image} of $\Phi$, that is the subgroupoid of $\Y$ whose set of objects is the image of $\Phi_0 \colon X_0 \to Y_0$ and whose set of arrows is the image of $\Phi_1 \colon X_1 \to Y_1$.  The groupoid $\Phi(\X)$ is endowed with the subspace topologies coming from $\Y$.  
		
		We will use $\overline{\Phi(\X)}$ to denote the \emph{full essential image} of $\Phi$, that is the full subgroupoid of $\Y$ whose set of objects is the orbit $\overline{\Phi(X_0)}$.  Once again, we endow $\overline{\Phi(\X)}$ with the subspace topologies.
	\end{df}
	
	\begin{lem}\label{surj-on-objs-implies-surj}
		Let $\Phi \colon \X \to \Y$ be a continuous functor between topological groupoids.  If $\Phi$ is surjective on objects, then $\Sh(\Phi)$ is a geometric surjection.
	\end{lem}
	\begin{proof}
		If $\Phi$ is surjective on objects, i.e.\ $\Phi_0$ is a surjection, then there is a commuting diagram of geometric morphisms
		\[
		\begin{tikzcd}
			\Sh(X_0) \ar[two heads]{r}{\Sh(\Phi_0)} \ar[two heads]{d}[']{u_\X} & \Sh(Y_0) \ar[two heads]{d}{u_\Y} \\
			\Sh(\X) \ar{r}{\Sh(\Phi)} & \Sh(\Y)
		\end{tikzcd}
		\]
		in which $\Sh(\Phi_0), u_\X, u_\Y$ are all geometric surjections.  It then follows from the right-cancellation property for geometric surjections (see \cite[\S A4.5]{elephant}) that $\Sh(\Phi)$ must be surjective as well.
	\end{proof}

	\begin{lem}\label{image-into-full-ess-image-is-surj}
		Let $\Phi \colon \X \to \Y$ be a continuous functor between topological groupoids where $\Y$ is logical.  The inclusion of the image $\Phi(\X)$ into the full essential image $\overline{\Phi(\X)}$ has Skula dense $\U$-orbits, for any open subgroupoid $\U \subseteq \Y$.
	\end{lem}
	\begin{proof}
		It is immediate that the inclusion $\Phi(\X) \subseteq \overline{\Phi(\X)}$ satisfies \cref{df:skula-dense} since, by definition, the orbit of $\Phi_0(X_0)$ in $\overline{\Phi(\X)}$ is the whole space of objects.
	\end{proof}

	\begin{coro}\label{surj-incl-fact-is-full-ess-image}
		Let $\Phi \colon \X \to \Y$ be a continuous functor between topological groupoids where $\Y$ is logical.  The full essential image factorisation $\X \to \overline{\Phi(\X)} \hookrightarrow \Y$ of a continuous functor $\Phi \colon \X \to \Y$ is sent by $\Sh$ to the surjection-inclusion factorisation of $\Sh(\Phi)$. 
	\end{coro}
	\begin{proof}
		We factor $\Phi \colon \X \to \Y$ as
		\[
		\begin{tikzcd}
			\X \ar[two heads]{r} & \Phi(\X) \ar[hook]{r} & \overline{\Phi(\X)} \ar[hook]{r} & \Y.
		\end{tikzcd}
		\]
		The first factor, $\X \twoheadrightarrow \Phi(\X)$, is surjective on objects and hence sent by $\Sh$ to a surjective geometric morphism by \cref{surj-on-objs-implies-surj}.  The second factor, $\Phi(\X) \hookrightarrow \overline{\Phi(\X)}$ is sent by $\Sh$ to a (localic) geometric surjection by \cref{image-into-full-ess-image-is-surj} and \cref{surjection-iff-skula-dense-orbits}.  Thus, the composite $\X \to \overline{\Phi(\X)}$ is sent by $\Sh$ to a geometric surjection.  It remains to show that the final factor $\overline{\Phi(\X)} \hookrightarrow \Y$ is sent to a geometric embedding.  This follows since the full essential image $\overline{\Phi(\X)}$ is a full replete subgroupoid of $\Y$ (see \cite{forssell-subgroupoid} or \cref{full-replete-yields-subtopos}).  
	\end{proof}

	\subsection{Geometric morphisms are induced by cospans}
	
Finally, we demonstrate that every geometric morphism between sheaf topoi is induced by a cospan of continuous functors of topological groupoids where one leg is a weak equivalence.  This is a key step in the proof of \cref{thm:main-thm-intro}.
	
\begin{prop}\label{geo_is_hom_of_grpd}
	Let $\X$ and $\Y$ be logical groupoids.  For any geometric morphism 
	\[\begin{tikzcd}
		f \colon \Sh(\X) \ar{r} & \Sh(\Y),
	\end{tikzcd}
	\] 
	there is a cospan of continuous functors
	\[
	\begin{tikzcd}
		\X \ar{r}{\Phi} & \W & \ar[hook']{l}[']{\Psi} \Y,
	\end{tikzcd}
	\]
	where $\Psi \in \weqv$, such that $\Sh(\Psi) \circ f$ is naturally isomorphic to $\Sh(\Phi)$.
\end{prop}
\begin{proof}
	In order to sidestep checking the continuity of $\Phi$ and $\Psi$, we argue in terms of classifying toposes and indexed groupoids of models, for which continuity will be immediate.  Let $\theory$ be a geometric theory, over a signature $\Sigma$, classified by the topos $\Sh(\Y)$.  By \cite[Theorem 7.1.5]{TST}, the geometric morphism $f$ is, up to isomorphism, induced by an \emph{expansion} $\theory'$ of $\theory$.  That is, $\theory'$ is a geometric theory over an expanded signature $\Sigma' \supseteq \Sigma$ (that potentially adds new sorts) which contains the axioms of $\theory$.  There is an evident functor on syntactic categories $\cat_\theory \to \cat_{\theory'}$, where an object $\form{\phi}{x}_\theory \in \cat_\theory$ is sent to $\form{\phi}{x}_{\theory'} \in \cat_{\theory'}$, which in turn induces the geometric morphism
	\[
	\begin{tikzcd}
		\topos_{\theory'} \simeq \Sh(\X) \ar{rr}{e^{\theory'}_\theory \simeq f} && \Sh(\Y) \simeq \topos_\theory.
	\end{tikzcd}
	\]
	Recall that, under the equivalence $\Sh(\Y) \simeq \topos_\theory$, the object $\form{\phi}{x}_\theory \in \cat_\theory \subseteq \topos_\theory$, for $\phi$ a geometric formula, is identified with the definable sheaf $\classv{\phi}{x}_\Y$ from \cref{ex:definable_sheaves}.

	By \cref{class-repr-grpds-for-theory}, the groupoid $\Y$ can be identified with a groupoid of set-based models of $\theory$, admitting some indexing $\Index \paronto \Y$, such that $Y_0$ and $Y_1$ are endowed with the respective logical topologies induced by this indexing, and moreover $\Y$ is conservative and eliminates parameters.  Similarly, the groupoid $\X$ can be identified with a groupoid of $\Index'$-indexed models of $\theory'$ that is conservative and eliminates parameters.

	Let $\W$ be the groupoid of all $\Index''$-indexed models of $\theory$, where $\Index''$ is chosen to be large enough that $\Index'' \supseteq \Index, \Index'$.  Equipped with this indexing, $\W$ eliminates parameters (see \cref{ex:elim_para}) and it is conservative since $Y_0 \subseteq W_0$.  Furthermore, there is an evident inclusion of groupoids ${\Psi}  \colon \Y \hookrightarrow\W$, which by \cref{ex:Morita-equivalent-subgroupoids-when-theory} is a weak equivalence.
	
	Consider the functor $ \Phi \colon \X \to \W$ which sends a $\theory'$-model $M \in X_0$ to its $\theory$-reduct $M|_\theory \in W_0$, endowed with the restriction of the indexing $\Index'' \supseteq \Index' \paronto M$, and an isomorphism of $\theory'$-models $M \xrightarrow{\alpha} N \in X_1$ to the isomorphism on the $\theory$-reducts $M|_\theory \xrightarrow{\alpha|_\theory} N|_\theory \in W_1$.  It immediately follows that these maps are continuous with respect to the relevant logical topologies.  
	
	We now have the desired cospan of continuous functors.  It remains to show that $\Sh(\Phi)$ and $\Sh(\Psi) \circ f $ are naturally isomorphic.  Since $\Sh(\Psi)$ is an equivalence, this is equivalent to showing a natural isomorphism $\Sh(\Phi) \simeq f \simeq e^{\theory'}_\theory$.  It suffices to demonstrate that the inverse images agree on the syntactic category $\cat_\theory \subseteq \topos_\theory$, viewed as a full subcategory of the classifying topos, since this generates the topos.  Applying $\Sh(\Phi)^\ast$ to the definable sheaf $\pi_{\classv{\phi}{x}}\colon \classv{\phi}{x}_\W \to W_0$, i.e.\ taking the pullback along the functor $\Phi \colon \X \to \W$, we obtain the sheaf over $\X$ whose points are pairs $\lrangle{\vec{a},M}$ where $\vec{a} \in \pi_{\classv{\phi}{x}}^{-1}(\phi_0(M))$, that is $\vec{a} \in \classv{\phi}{x}_{M|_\theory} = \classv{\phi}{x}_M$.  It is thus easily confirmed that $\Sh(\Phi)^\ast(\classv{\phi}{x}_\W) \cong \classv{\phi}{x}_\X$, which corresponds to $\form{\phi}{x}_{\theory'}$ under the equivalence $\Sh(\X) \simeq \topos_{\theory'}$.  This witnesses the natural isomorphism $\Sh(\Phi) \simeq e^{\theory'}_\theory$.
\end{proof}

\begin{rem}\label{rem:geo-is-hom}
	It is worth emphasising the philosophical content of \cref{geo_is_hom_of_grpd}.  We have shown that every geometric morphism $f \colon \topos \to \ftopos$ between topoi with enough points is of the form $\Sh(\Phi) \colon \Sh(\X) \simeq \topos \to \ftopos \simeq \Sh(\W)$ for some continuous functor $\Phi \colon \X \to \W$ between logical groupoids.  That is to say, the geometric morphism $f$ is determined by a suitably continuous functor from a small groupoid of points of $\topos$ to a small groupoid of points for $\ftopos$, just like how continuous maps of topological spaces are, tautologically, a continuous assignment of points.  This was partly the motivation for the introduction of \emph{ionaid} in \cite{ionad}.  A similar exposition involving \emph{generalised} (i.e.\ not set-based) points can be found in \cite{vickers}.
\end{rem}


	\section{\'{E}tale complete topological groupoids}\label{sec:ectgrpd}
	
	While the theory of weak equivalences for topological groupoids established in the preceding section requires only that the topological groupoids be \emph{logical}.  In order to establish our desired bi-equivalence \cref{thm:main-thm-intro}, we will need to further restrict our attention to those logical groupoids whose space of arrows conforms well with the 2-cell data of its topos of sheaves.
	\begin{enumerate}
		\item We first introduce this restricted class of logical groupoids, the \emph{étale complete topological groupoids}.  We then give a logical characterisation in \cref{coro:etale_compl_characterisation}.
		\item Next, we demonstrate that the restriction of the bi-functor $\Sh \colon \TopGrpd \to \Topos_\wep^\iso$ to étale complete groupoids is full and faithful on 2-cells.
		\item Finally, we observe that étale complete topological groupoids forms a reflective (bi-)subcategory of (sober) logical groupoids.
	\end{enumerate}

	\begin{df}\label{df:etale_complete}
		An open topological groupoid $\X = (X_1 \rightrightarrows X_0)$ is said to be \emph{étale complete} if $X_0$ and $X_1$ are both sober spaces and the square
		\[
		\begin{tikzcd}
			\Sh(X_1) \ar{r}{\Sh(s)} \ar{d}[']{\Sh(t)} & \Sh(X_0) \ar{d}{u_\X} \\
			\Sh(X_0) \ar{r}{u_\X} & \Sh(\X)
		\end{tikzcd}
		\]
		is a bi-pullback in the bi-category $\Topos_\WEP$.  Let $\ECTGrpd \subseteq \TopGrpd$ denote the full bi-subcategory of étale complete open topological groupoids.
	\end{df}

	\begin{rem}
		Our choice of terminology, \emph{étale complete}, was chosen by analogy with the étale complete localic groupoids of Moerdijk \cite[Definition 7.2]{cont1}.  \'{E}tale complete topological groupoids were also discussed in \cite[\S 8.3]{myselfgrpds}.  Note that the use of the term \emph{\'{e}tale complete} in \cite{myselfgrpds} differs slightly from our use here in that \cite{myselfgrpds} did not require that the space of objects $X_0$ and arrows $X_1$ be sober spaces (cf.\ \cite[Definition 8.16]{myselfgrpds} and \cref{coro:etale_compl_characterisation}).
	\end{rem}
	
	Checking étale completeness can be broken down into two conditions without the need to compute bi-pullbacks of topoi with enough points explicitly.  Eventually, this will allow us to give a complete logical description of all the \'{e}tale complete topological groupoids in \cref{coro:etale_compl_characterisation}.
	
	\begin{prop}\label{etale_complete_iff_logical+stuff}
		An open topological groupoid $\Xtt = (X_1^{\tau_1}\rightrightarrows X_0^{\tau_0})$, where $X_0^{\tau_0}$ and $X_1^{\tau_1}$ are both sober spaces, is étale complete if and only if:
		\begin{enumerate}
			\item $\Xtt$ is logical in the sense of \cref{df:loggrpd}
			\item  and, for any pair $x,y \in X_0$, an isomorphism of the corresponding geometric morphisms
			\[\begin{tikzcd}
				\sets & {\Sh(X_0^{\tau_0})} \\
				{\Sh(X_0^{\tau_0})} & {\Sh(\Xtt)}
				\arrow["u_\X", from=1-2, to=2-2]
				\arrow["u_\X"', from=2-1, to=2-2]
				\arrow["x", from=1-1, to=1-2]
				\arrow["y"', from=1-1, to=2-1]
				\arrow["\alpha"', Rightarrow, from=1-2, to=2-1]
				\arrow["\sim"{marking, allow upside down}, shift left=2, draw=none, from=1-2, to=2-1]
			\end{tikzcd}\]
			is instantiated by an arrow $x \xrightarrow{\alpha} y \in X_1$.
		\end{enumerate}
	\end{prop}
	
	Before proving \cref{etale_complete_iff_logical+stuff}, we first show that $\Topos_\wep$ inherits an important property from the bi-category $\Topos$.

	\begin{lem}\label{localic-stable-in-Topos_wep}
		In the bi-category $\Topos_\wep$, localic geometric morphisms are stable under bi-pullback.
	\end{lem}
	\begin{proof}
		This essentially follows from the fact that localic geometric morphisms are stable under bi-pullback in $\Topos$ \cite[Proposition 2.1]{fact1}.  Recall that the bi-category of topoi with enough points forms a co-reflective bi-subcategory of the bi-category of all topoi,
		\[\begin{tikzcd}
			{\Topos_\WEP} && \Topos.
			\arrow[""{name=0, anchor=center, inner sep=0}, shift left=2, hook, from=1-1, to=1-3]
			\arrow[""{name=1, anchor=center, inner sep=0}, "\Pt", shift left=2, from=1-3, to=1-1]
			\arrow["\dashv"{anchor=center, rotate=-90}, draw=none, from=0, to=1]
		\end{tikzcd}\]
		This is because each topos $\topos$ has a largest subtopos with enough points $\Pt(\topos) \hookrightarrow \topos$ through which every geometric morphism $\ftopos \to \topos$ whose domain has enough points must factor (see \cite[Corollary 7.18]{topos}).  Thus, $\Pt$ preserves bi-limits, in particular bi-pullbacks.  So the bi-pullback in $\Topos_\wep$ of a geometric morphism $f$ along $g$, as in the cospan
		\[
		\begin{tikzcd}
			& \ftopos \ar{d}{f} \\
			\mathcal{G} \ar{r}{g} & \topos,
		\end{tikzcd}
		\]
		can be calculated as the composite $\Pt(\ftopos \times_\topos \mathcal{G}) \hookrightarrow \ftopos \times_\topos \mathcal{G} \to \mathcal{G}$, where $\ftopos \times_\topos \mathcal{G}$ denotes the bi-pullback in $\Topos$.  If $f$ is localic, then so is $\ftopos \times_\topos \mathcal{G} \to \mathcal{G}$ and hence so is the composite $\Pt(\ftopos \times_\topos \mathcal{G}) \hookrightarrow \ftopos \times_\topos \mathcal{G} \to \mathcal{G}$, since inclusions of subtopoi are localic, completing the proof.
	\end{proof}

	\begin{proof}[Proof of \cref{etale_complete_iff_logical+stuff}]
		Suppose first that $\Xtt$ is étale complete.  By the universal property of the bi-pullback, every isomorphism of geometric morphisms $\alpha \colon u_\X \circ x \Rightarrow u_\X \circ y$, for $x, y \in X_0$, yields a geometric morphism $\alpha \colon \sets \to \Sh(X_1^{\tau_0})$, and hence a point of $X_1$.  If $\sigma$ is another topology on $X_1$ for which $\X_{\tau_0}^{\sigma}$ is a topological groupoid with the same topos of sheaves as $\Xtt$, then there is a factoring geometric morphism
		\[
		\begin{tikzcd}
			\Sh(X_1^{\sigma}) \ar[dashed]{rd} \ar[bend left = 2em]{rrd} \ar[bend right=3em]{ddr} && \\
			& \Sh(X_1^{\tau_1}) \ar{r}{\Sh(t)} \ar{d}[']{\Sh(s)} \ar[draw = none]{rd}[anchor = center, pos=0.125]{\lrcorner} & \Sh(X_0^{\tau_0}) \ar{d}{u_\X} \\
			& \Sh(X_0^{\tau_0}) \ar{r}{u_\X} & \Sh(\Xtt) &[-30pt] \simeq \Sh(\X_{\tau_0}^{\sigma}),
		\end{tikzcd}
		\]
		where the induced continuous map $X_1^{\sigma} \to X_1^{\tau_1}$ acts as the identity on the points of $X_1$, and so the inclusion $\tau_1 \subseteq \sigma$ is forced.  Thus, $\Xtt$ is both logical and every natural isomorphism between $u_\X \circ x$ and $u_\X \circ y$ is instantiated in the space of arrows.
		
		Now suppose that $\Xtt$ is logical and that every natural isomorphism $u_\X \circ x \cong u_\X \circ y$ is instantiated in the space of arrows.  We first demonstrate that the bi-pullback of the cospan
		\begin{equation}\label{prop:char_of_etalecomplete:eq:cospan-of-ux}
			\begin{tikzcd}
				& \Sh(X^{\tau_0}) \ar{d}{u_\X} \\
				\Sh(X^{\tau_0}) \ar{r}{u_\X} & \Sh(\Xtt)
			\end{tikzcd}
		\end{equation}
		is the topos of sheaves on a topological space whose underlying set of points is $X_1$, before demonstrating that the topology must be given by $\tau_1$.  For the former, by \cref{localic-stable-in-Topos_wep} we know that the bi-pullback of \cref{prop:char_of_etalecomplete:eq:cospan-of-ux} is a localic topos with enough points, i.e.\ the sheaves on a topological space.  The universal property of the bi-pullback entails that the points of this space are in bijective correspondence with the natural isomorphism $u_\X \circ x \cong u_\X \circ y$, which is precisely the set $X_1$ by assumption.  Hence, the bi-pullback of \cref{prop:char_of_etalecomplete:eq:cospan-of-ux} is given by $\Sh(X_1^\sigma)$ for some topology $\sigma$ on $X_1$.  It is also easily seen, by a diagram chase, that the projections $\Sh(X_1^\sigma) \rightrightarrows \Sh(X_0^{\tau_0})$ act as $s,t \colon X_1 \rightrightarrows X_0$ on the underlying points of $X_1$.
		
		It now remains to show that the topology $\sigma$ is the topology $\tau_1$.  One inclusion, $\sigma \subseteq \tau_1$, follows because $\Sh(X_1^{\tau_1})$ is a bi-cone for the cospan \cref{prop:char_of_etalecomplete:eq:cospan-of-ux}, and the universally induced arrow $\Sh(X_1^{\tau_1}) \to \Sh(X_1^\sigma)$ acts as the identity on points.  We now show the reverse inclusion.
		
		Let $\theory$ be a theory classified by $\Sh(\Xtt)$, so that $\X$ can be identified with a groupoid of models for $\theory$ and $\tau_1$ and $\tau_0$ with the logical topologies, given by some indexing $\Index \paronto \X$.  By the fact that both $ s,t \colon X_1^\sigma \rightrightarrows X_0^{\tau_0}$ are continuous, we have that $s^{-1}(\classv{\phi}{m}_\X), t^{-1}(\classv{\phi}{m}_\X) \in \sigma$, for each formula $\phi$ and tuple of parameters $\vec{m} \in \Index$, and so it remains only to show that $\class{\vec{m}_1 \mapsto \vec{m}_2}_\X \in \sigma$ too.
		
		Consider the canonical natural isomorphism filling in the square
		\[\begin{tikzcd}
			\Sh(X_1^{\sigma}) & {\Sh(X_0^{\tau_0})} \\
			{\Sh(X_0^{\tau_0})} & {\Sh(\Xtt)}
			\arrow["u_\X", from=1-2, to=2-2]
			\arrow["u_\X"', from=2-1, to=2-2]
			\arrow["\Sh(s)", from=1-1, to=1-2]
			\arrow["\Sh(t)"', from=1-1, to=2-1]
			\arrow["\Theta"', Rightarrow, from=1-2, to=2-1]
			\arrow["\sim"{marking, allow upside down}, shift left=2, draw=none, from=1-2, to=2-1]
		\end{tikzcd}\]
		The component of $\Theta$ at $W \in \Sh(\Xtt)$ is a homeomorphism
		\[
		\begin{tikzcd}
			\lrset{(\alpha,w)}{s(\alpha) = q(w)} &[-30pt] \cong \Sh(s)^\ast W \ar{rr}{\Theta_W} \ar[phantom, shift right = 2]{rr}[']{\sim} \ar[bend right=1.5em]{dr}[']{\pr} && \Sh(t)^\ast W \ar[bend left = 1.5em]{dl}{\pr'} &[-30pt] \cong \lrset{(\alpha,w')}{t(\alpha) = q(w')} \\
			&& X_1 &&
		\end{tikzcd}
		\]
		and acts by $(\alpha,w) \mapsto (\alpha, \alpha \cdot w)$.
		
		Consider the case where $W$ is given by a definable sheaf $\pi_{\classv{\top}{x}} \colon \classv{\top}{x}_\X \to X_0$ from example \cref{ex:definable_sheaves}.  Then 
		\[\lrset{\left(M \xrightarrow{\alpha} N, \vec{x} \in M\right)}{\vec{x} = \vec{m}_1} = X_1 \times_{X_0} \class{\vec{x} = \vec{m}_1}_\X \subseteq \Sh(s)^\ast\classv{\top}{x}_\X\] 
		is an open subset, as is
		\[
		\lrset{\left(M \xrightarrow{\alpha} N, \vec{x} \in M\right)}{\vec{x} = \vec{m}_2} = X_1 \times_{X_0} \class{\vec{x} = \vec{m}_2}_\X \subseteq \Sh(t)^\ast\classv{\top}{x}_\X
		\]
		Thus, as $\Theta_W$ must be continuous, the subset
		\begin{align*}
			\lrset{\left(M \xrightarrow{\alpha} N, \vec{x} \in M\right)}{\vec{x} = \vec{m}_1 \text{ and } \alpha(\vec{x}) = \vec{m}_2 } & = X_1 \times_{X_0} \class{\vec{x} = \vec{m}_1}_\X \cap \Theta_W^{-1}(X_1 \times_{X_0} \class{\vec{x}=\vec{m}_2}_\X) \\
			& \subseteq \Sh(s)^\ast \classv{\top}{x}_\X
		\end{align*}
		is open.  Finally, being a local homeomorphism, the projection $\pr_1 \colon \Sh(s)^\ast \classv{\top}{x}_\X \to X_1^\sigma$ is, in particular, open.  Hence, the subset
		\begin{align*}
			\class{\vec{m}_1 \mapsto \vec{m}_2}_\X & = \lrset{M \xrightarrow{\alpha} N }{\alpha(\vec{m}_1) = \vec{m}_2}, \\
			& = \pr_1\left( X_1 \times_{X_0} \class{\vec{x} = \vec{m}_1}_\X \cap \Theta_W^{-1}(X_1 \times_{X_0} \class{\vec{x}=\vec{m}_2}_\X)\right) \subseteq X_1^\sigma
		\end{align*}
		is open.  From this, we deduce that $\sigma = \tau_1$ as desired.
	\end{proof}

	\begin{lem}\label{sober-objs-and-elim-para-implies-sober-arrs}
		Let $\theory$ be a theory and $\Xtt = (X_1^{\tau_1} \rightrightarrows X_0^{\tau_0})$ a groupoid of models for $\theory$ given the logical topologies for some indexing $\Index \paronto \X$.  If the space $X_0^{\tau_0}$ is sober, and $X_1$ contains all isomorphisms of models contained in $X_0$, then $X_1^{\tau_1}$ is sober too.
	\end{lem}
	\begin{proof}
		Let $F \subseteq \tau_1$ be a completely prime filter.  We must show that $F$ is the neighbourhood filter for a unique point of $X_1$.  The sets $ \lrset{U \in \tau_0}{s^{-1}(U) \in F}$ and $ \lrset{V \in \tau_0}{t^{-1}(V) \in F}$ are both completely prime filters of $\tau_0$, and so, by the sobriety of $X_0^{\tau_0}$, they correspond to unique models $M, N \in X_0$.  We will show that $F$ is the neighbourhood filter for a unique isomorphism $M \xrightarrow{\alpha} N \in X_1$.
		
		For an element $n \in M$, let $m \in \Index$ be a parameter indexing $n$.  We therefore have that $M \in \class{m : \top}_\X$ and so $s^{-1}(\class{m : \top}_\X) \in F$.  Then $ s^{-1}(\class{m : \top}_\X) \subseteq \bigcup_{m' \in \Index} \class{m \mapsto m'}_\X \in F$, and so, as $F$ is completely prime, $\class{m \mapsto m'}_\X$ must be in $F$ for some $m'$.  Now we have that $\class{m \mapsto m'}_\X \subseteq t^{-1}(\class{m' : \top}_\X)$.  Therefore, $N \in \class{m' : \top}_\X$, i.e.\ there is some $n' \in N$ indexed by $m'$.  We define $\alpha(n) = n'$.  This yields a well defined map $M \xrightarrow{\alpha} N$ because for any other parameter $m'' \in \Index$ for which $\class{m \mapsto m''}_\X \in F$, using that $F$ is a completely prime filter, it is possible to show that $n'$ must also be indexed by $m''$.  In a similar fashion, we can construct a well-defined map $M \xrightarrow{\alpha^{-1}} N$ which is an inverse to $\alpha$.  
		
		The map $\alpha$ preserves the interpretation of formulae, and hence it is an isomorphism of models.  To see why this is the case, if $M \vDash \phi(\vec{n}) $, then, letting $\vec{m} \in \Index$ be a tuple of parameters indexing $\vec{n}$, we have by the above construction a tuple $\vec{m}' \in \Index$ indexing $\alpha(\vec{n})$ such that $s^{-1}(\classv{\phi}{m}_\X) \cap \class{\vec{m} \mapsto \vec{m}'}_\X \in F$.  Now, since $s^{-1}(\classv{\phi}{m}_\X) \cap \class{\vec{m} \mapsto \vec{m}'}_\X \subseteq t^{-1}(\class{\vec{m}' : \phi}_\X)$, we conclude that $N \in \class{\vec{m}' : \phi}_\X$, i.e.\ $N \vDash \phi(\alpha(\vec{n}))$.  As $X_1$ contains all possible isomorphisms of models contained in $X_0$, we have that $\alpha \in X_1$.
		
		Suppose that $M \xrightarrow{\gamma} N \in X_1$ is another isomorphism of models whose neighbourhood filter is $F$.  Then $\alpha \in \class{m \mapsto m'}_\X \iff \gamma \in \class{m \mapsto m'}_\X$, and so we conclude that $\alpha(n) = \gamma(n)$ for all $n \in M$.
	\end{proof}
	
	\begin{coro}\label{coro:etale_compl_characterisation}
		The \'{e}tale complete topological groupoids are precisely those topological groupoids $\Xtt = (X_1^{\tau_1} \rightrightarrows X_0^{\tau_0})$ coming from a groupoid of models for some theory $\theory$, where $\tau_1$ and $\tau_0$ are both the logical topologies given some indexing $\Index \paronto \X$, where:
		\begin{enumerate}
			\item\label{enum:etale-compl:elim_para} $\X$ eliminates parameters for this indexing,
			\item\label{enum:etale-compl:sober} the space $X_0^{\tau_0}$ is sober,
			\item\label{enum:etale-compl:all_isos} the set $X_1$ contains all isomorphisms between models contained in $X_0$.
		\end{enumerate}
	\end{coro}
	\begin{proof}
		If $\Xtt$ is \'{e}tale complete then, being logical, $\Xtt$ is a groupoid of indexed models for some theory $\theory$ such that $\Xtt$ eliminates parameters and both $\tau_0$ and $\tau_1$ are the induced logical topologies.  By assumption, $X_0^{\tau_0}$ is sober, and by \cref{etale_complete_iff_logical+stuff} every isomorphism between points of $X_0$, i.e.\ models, is instantiated in the set of arrows.
		
		For the converse, a groupoid $\Xtt$ of models for a theory $\theory$, endowed with the logical topologies for an indexing $\Index \paronto \X$, and satisfying conditions \cref{enum:etale-compl:elim_para} to \cref{enum:etale-compl:all_isos}, is open by \cite[Lemma 7.3]{myselfgrpds} and moreover logical by \cref{class-of-loggrpd}, and the space $X_1^{\tau_1}$ is sober by \cref{sober-objs-and-elim-para-implies-sober-arrs}.  Finally, as $X_1$ also contains every isomorphism of models contained in $X_0$, the hypotheses of \cref{etale_complete_iff_logical+stuff} are satisfied, and so $\Xtt$ is \'{e}tale complete.
	\end{proof}
	
	\begin{ex}\label{ex:forssell-groupoids-are-ec}
		In particular, the groupoid of all $\Index$-indexed models for a theory $\theory$, and all isomorphisms of these models, satisfies the conditions of \cref{coro:etale_compl_characterisation} (see \cite[Proposition 1.2.7]{awodeyforssell}).
	\end{ex}

	\begin{ex}\label{ex:etale-complete-groups}
		Using \cite[Theorem 4.14]{hodges}, a subgroup $G \subseteq \Omega(X)$ of the topological permutation group on a set $X$ is étale complete if and only if $G$ is a closed subgroup.  Combined with \cref{groups-thick-becomes-dense}, this entails that the only weak equivalences between étale complete topological groups are homeomorphisms.
	\end{ex}
	
	\paragraph{Transformations for \'{e}tale complete topological groupoids.}
	
	Below in \cref{prop:weqv_is_calc}, we will show that $\weqv$ defines a left bi-calculus of fractions on $\ECTGrpd$.  The localization bi-functor $\ECTGrpd \to [\weqv^{-1}]\ECTGrpd$, as constructed in \cite{pronk}, is full and faithful on 2-cells.  Therefore, our desired bi-equivalence 
	\[[\weqv^{-1}]\ECTGrpd \simeq \Topos_\WEP^\iso\]
	is induced by the bi-functor $\Sh$ only if its restriction to étale complete topological groupoids is full and faithful on 2-cells.  This is provided by the following lemma.
	
	\begin{lem}\label{Sh-ff-on-2-cells}
		For a pair of \'etale complete topological groupoids $\X, \Y$, the functor on hom-categories
		\[
		\Sh \colon \ECTGrpd(\X, \Y) \to \Topos_\wep^\iso(\Sh(\X),\Sh(\Y))
		\]
		is full and faithful.
	\end{lem}
	\begin{proof}
		Let $\Phi, \Psi \colon \X \rightrightarrows \Y$ be a pair of continuous functors.  Note that, for each $x \in X_0$, an arrow $\Phi_0(x) \xrightarrow{\alpha} \Psi_0(x) \in Y_1$ corresponds uniquely with a natural isomorphism
		\[
		\begin{tikzcd}
			\sets \ar{rr}{x} \ar{dd}[']{x} && \Sh(X_0) \ar{d}{\Sh(\Phi_0)} \ar[Rightarrow]{lldd}[']{\alpha}  \ar[draw=none, shift left=2]{ldld}[marking]{\sim} \\
			&& \Sh(Y_0) \ar{d}{u_\Y} \\
			\Sh(X_0) \ar{r}[']{\Sh(\Psi_0)} & \Sh(Y_0) \ar{r}[']{u_\Y} & \Sh(\Y).
		\end{tikzcd}
		\]
		In particular, if $a \colon \Phi \Rightarrow \Psi$ is a continuous transformation, then the component $\Phi_0(x) \xrightarrow{a_x} \Psi_0(x)$ corresponds to the natural isomorphism
		\[\begin{tikzcd}
			\sets &[-10pt] &[20pt] {\Sh(X_0)} \\[-10pt]
			&& {\Sh(\X)} \\[20pt]
			{\Sh(X_0)} & {\Sh(\X)} & {\Sh(\Y)}.
			\arrow["x", from=1-1, to=1-3]
			\arrow["x"', from=1-1, to=3-1]
			\arrow["{u_\X}", from=1-3, to=2-3]
			\arrow[equal, from=2-3, to=3-2]
			\arrow[""{name=0, anchor=center, inner sep=0}, "{\Sh(\Phi)}", from=2-3, to=3-3]
			\arrow["{u_\X}", from=3-1, to=3-2]
			\arrow[""{name=1, anchor=center, inner sep=0}, "{\Sh(\Psi)}"', from=3-2, to=3-3]
			\arrow["{\Sh(a)}"', shorten <=5pt, shorten >=5pt, Rightarrow, from=0, to=1]
			\arrow["\sim"{marking, allow upside down}, shift left=2, draw=none, from=0, to=1]
		\end{tikzcd}\]
		So if $\Sh(a) = \Sh(a')$, for another transformation $a' \colon \Phi \Rightarrow \Psi$, then $a_x = a'_x$ for all $x \in X_0$, i.e.\ $a= a'$.  Thus, $\Sh \colon \ECTGrpd(\X, \Y) \to \Topos_\wep^\iso(\Sh(\X),\Sh(\Y))$ is faithful.
		
		Now let $\gamma \colon \Sh(\Phi) \Rightarrow \Sh(\Psi)$ be a natural isomorphism between the induced geometric morphisms.  Then there is a natural isomorphism
		\[
		u_\Y \circ \Sh(\Phi_0) = \Sh(\Phi) \circ u_\X \xRightarrow{\gamma \ast u_\X} \Sh(\Psi) \circ u_\X = u_\Y \circ \Sh(\Psi_0),
		\]
		i.e.\ filling the square
		\[
		\begin{tikzcd}
			\Sh(X_0) \ar{r}{\Sh(\Phi_0)} \ar{d}[']{\Sh(\Psi_0)} & \Sh(Y_0) \ar{d}{u_\Y} \ar[Rightarrow]{ld}[']{\gamma \ast u_\X}  \ar[draw=none, shift left=2]{ld}[marking]{\sim}\\
			\Sh(Y_0) \ar{r}[']{u_\Y} & \Sh(\Y).
		\end{tikzcd}
		\]
		Therefore, by the universal property of the bi-pullback, there is an induced geometric morphism $\Sh(X_0) \to \Sh(Y_1)$, which by sobriety, corresponds to a continuous map $a \colon X_0 \to Y_1$.  By the imposed commutativity conditions, $a$ encodes a natural transformation $a \colon \Phi \Rightarrow \Psi$.  Moreover, by the way $a$ was obtained, for each point of $X_0$, i.e.\ each geometric morphism $x \colon \sets \to \Sh(X_0)$, we have that 
		\[\Sh(a) \ast (u_\X \circ x) = \gamma \ast (u_\X \circ x).\]
		Since the family of geometric morphisms $\lrset{u_\X \circ x}{x \in X_0}$ is jointly surjective, it follows that $\Sh(a) = \gamma$.  Hence, the functor $\Sh \colon \ECTGrpd(\X, \Y) \to \Topos_\wep^\iso(\Sh(\X),\Sh(\Y))$ is full.
	\end{proof}

	\paragraph{The étale completion of a sober logical groupoid.}
	
	We will say that a logical groupoid $\X = (X_1 \rightrightarrows X_0)$ is \emph{sober} if the space of objects $X_0$ is sober.  Let $\LogGrpd_{\mathrm{sob}}$ denote the bi-subcategory of $\LogGrpd$ on sober logical groupoids.  Recall from \cite[\S 8.3]{myselfgrpds} that there is a left bi-adjoint 
	\[
	\begin{tikzcd}
		{\LogGrpd_{\mathrm{sob}}} && \ECTGrpd.
		\arrow[""{name=0, anchor=center, inner sep=0}, "{\hat{(-)}}", shift left=3, from=1-1, to=1-3]
		\arrow[""{name=1, anchor=center, inner sep=0}, shift left=3, hook', from=1-3, to=1-1]
		\arrow["\dashv"{anchor=center, rotate=-90}, draw=none, from=0, to=1]
	\end{tikzcd}
	\]
	Given a sober logical groupoid $\X = (X_1 \rightrightarrows X_0)$, its \emph{\'{e}tale completion} $\hat{\X}$ can be given two equivalent definitions.
	\begin{enumerate}
		\item Firstly, $\hat{\X} = (\hat{X}_1 \rightrightarrows X_0)$ is the topological groupoid whose space of objects remains the same as in $\X$, but where the space of arrows $\hat{X}_1$ is computed as the bi-pullback
		\[
		\begin{tikzcd}
			\Sh(\hat{X}_1) \ar{r} \ar{d} \ar[draw = none]{rd}[anchor = center, pos=0.125]{\lrcorner} & \Sh(X_0) \ar{d}{u_\X} \\
			\Sh(X_0) \ar{r}{u_\X} & \Sh(\X).
		\end{tikzcd}
		\]
		\item Alternatively, we recall that $\X$ can be viewed as a representing groupoid of models for a geometric theory $\theory$ (in particular, $\X$ eliminates parameters for some indexing).  Then $\hat{\X}$ is the groupoid of $\theory$-models whose objects are the same models that are included in $\X$ (endowed with the same indexing) but whose arrows are \emph{all} $\theory$-model isomorphisms between these models (and the topologies on $\hat{\X}$ are the corresponding logical topologies for the given indexing).
	\end{enumerate}

	\begin{prop}
		The inclusion of a sober logical groupoid $\X$ into its \'{e}tale completion $\hat{\X}$ is a weak equivalence.
	\end{prop}
	\begin{proof}
		By \cite[Corollary 8.20]{myselfgrpds}, the inclusion $\X \subseteq \hat{\X}$ induces an equivalence of topoi $\Sh(\X) \simeq \Sh(\hat{\X})$, and so we apply \cref{classification-of-weqv}.
	\end{proof}

	\section{The bi-equivalence}\label{sec:bieqv}


		We have now collected enough results to prove our main result, the bi-equivalence stated in \cref{thm:main-thm-intro}.  This will be achieved via an application of the following result by Pronk \cite{pronk}, a bi-categorical extension to Gabriel and Zisman's localisation result (\cite[Proposition I.1.3]{gabrielzisman}). 
	
	\begin{lem}[Proposition 24 \cite{pronk}]\label{lem:equiv_with_left_fracs}
		Let $G \colon \cat \to \dcat$ be a bi-functor and let $\Sigma$ be a class of 1-morphisms in $\cat$ admitting a left bi-calculus of fractions.  Suppose that
		\begin{enumerate}
			\item the bi-functor $G$ is essentially surjective on objects and fully faithful on 2-cells,
			\item for each $f \in \Sigma$, $G(f)$ is an equivalence,
			\item and for any arrow $G(c) \xrightarrow{g} G(d) \in \dcat$, there are a pair of arrows $c \xrightarrow{f} e \in \cat$ and $d \xrightarrow{\sigma} e \in \Sigma$ such that the triangle
			\[
			\begin{tikzcd}
				G(c) \ar{r}{g} \ar{rd}[']{G(f)} & G(d) \ar{d}{G(\sigma)} \\
				& G(e)
			\end{tikzcd}
			\]
			commutes up to isomorphism.
		\end{enumerate}
		Then there is a bi-equivalence of bi-categories $[\Sigma^{-1}]\cat \simeq \dcat$.
	\end{lem}

	\begin{prop}\label{prop:weqv_is_calc}
		The class $\weqv$ of weak equivalences defines a left bi-calculus of fractions on $\ECTGrpd$.
	\end{prop}
	\begin{proof}
		Recall from \cite[\S 2.1]{pronk} that there are three conditions required to be a left bi-calculus of fractions.  The class $\weqv$ must be wide, satisfy the left Ore condition, and the left cancellability condition (all in the bi-categorical sense).
		\begin{enumerate}
			\item (Wideness) Wideness has already been shown in \cref{weqv-is-wide}.
			
			\item (Left Ore condition) Let
			\[
			\begin{tikzcd}
				\Y \ar[hook]{r}{\Psi} \ar{d}[']{\Phi} & \W \\
				\X 
			\end{tikzcd}
			\]
			be a span where $\Psi \in \weqv$.  We wish to find a square of continuous functors
			\begin{equation}\label{prop:weqv_is_calc:eq:square}
				\begin{tikzcd}
					\Y \ar[hook]{r}{\Psi} \ar{d}[']{\Phi} \ar[phantom]{rd}[description]{\cong}  & \W \ar{d}{\Phi'} \\
					\X \ar[hook]{r}{\Psi'} & \V
				\end{tikzcd}
			\end{equation}
			that commutes up to isomorphism, where $\Psi' \in \weqv$.  Consider the induced geometric morphism
			\[
			\begin{tikzcd}
				\Sh(\W) \ar{r}{\Sh(\Psi)^{-1}} \ar[phantom, shift right=0.5em]{r}[']{\sim} & \Sh(\Y) \ar{r}{\Sh(\Phi)} & \Sh(\X).
			\end{tikzcd}
			\]
			By applying \cref{geo_is_hom_of_grpd}, there is a cospan
			\[
			\begin{tikzcd}
				& \W \ar{d}{\Phi'} \\
				\X \ar[hook]{r}{\Psi'} & \V
			\end{tikzcd}
			\]
			such that $\Psi' \in \weqv$ and $\Sh(\Phi') \cong \Sh(\Psi') \circ \Sh(\Phi) \circ \Sh(\Psi)^{-1}$.  Now since $\Sh$ is fully faithful on 2-cells by \cref{Sh-ff-on-2-cells}, the isomorphism $\Sh(\Phi') \circ \Sh(\Psi) \cong \Sh(\Psi')\circ \Sh(\Phi)$ yields a continuous isomorphism $\Phi' \circ \Psi \cong \Psi' \circ \Phi$ as desired.
			
			\item (Left cancellability) Let
			\[
			\begin{tikzcd}
				\X \ar{r}{\Xi} & \Y \ar[shift left=2]{r}{\Phi} \ar[shift right=2]{r}[']{\Psi} & \W
			\end{tikzcd}
			\]
			be a fork of continuous functors of logical groupoids that commute up to continuous isomorphism, and where $\Xi \colon \X \to \Y \in \weqv$.  We wish to find another continuous functor $\Xi' \colon \W \to \V \in \weqv$ such that the fork
			\[
			\begin{tikzcd}
				\Y \ar[shift left=2]{r}{\Phi} \ar[shift right=2]{r}[']{\Psi} & \W \ar{r}{\Xi'} & \V
			\end{tikzcd}
			\]
			also commutes up to continuous isomorphism.  Since the induced fork of geometric morphisms
			\[
			\begin{tikzcd}
				\Sh(\X) \ar{r}{\Sh(\Xi)} \ar[phantom, shift right=0.5em]{r}[']{\sim} & \Sh(\Y) \ar[shift left=2]{r}{\Sh(\Phi)} \ar[shift right=2]{r}[']{\Sh(\Psi)} & \Sh(\W)
			\end{tikzcd}
			\]
			also commutes up to isomorphism, and so $\Sh(\Phi) \cong \Sh(\Psi)$, by \cref{Sh-ff-on-2-cells} we have that $\Phi \cong \Psi$, and so we can simply take $\id_\W$ as the desired weak equivalence.
			
			We must also show that, for any other weak equivalence of logical groupoids $\Xi'' \colon \W \to\V \in \weqv$ that makes the fork
			\[
			\begin{tikzcd}
				\Y \ar[shift left=2]{r}{\Phi} \ar[shift right=2]{r}[']{\Psi} & \W \ar{r}{\Xi''} & \V
			\end{tikzcd}
			\]
			commute up to isomorphism, then there are weak equivalences ${\mathrm{P}} \colon \W \to \V', \mathrm{P}' \colon \V \to \V' \in \weqv$ and a coherent choice of isomorphism $\mathrm{P} \circ \id_\W \xRightarrow{\sim} \mathrm{P}' \circ \Xi''$.  But we simply take $\mathrm{P}$ as $\Xi'' $ and $\mathrm{P}'$ as $\id_\V$.  It is trivial to show that the identity transformation $\Xi'' \circ \id_\W = \id_\V \circ \Xi''$ satisfies the necessary coherence condition expressed in \cite[\S 2.1]{pronk}.
		\end{enumerate} 
		Thus, $\weqv$ defines a left bi-calculus of fractions on $\ECTGrpd$.
	\end{proof}

	\begin{repeated-theorem}[\cref{thm:main-thm-intro}]
		There is a bi-equivalence
		\[
		[\weqv^{-1}]\ECTGrpd \simeq \Topos_\wep^\iso.
		\]
	\end{repeated-theorem}
	\begin{proof}
		The functor $\Sh \colon \ECTGrpd \to \Topos_\wep^\iso$ is essentially surjective by \cite{BM} and fully faithful on 2-cells by \cref{Sh-ff-on-2-cells}.  Every weak equivalence $\Y \hookrightarrow \X$ is sent by $\Sh$ to an equivalence of topoi by \cref{classification-of-weqv}.  By \cref{geo_is_hom_of_grpd}, every geometric morphism $f \colon \Sh(\X) \to \Sh(\Y)$ is induced, up to isomorphism, by a cospan $\X \to \W \hookleftarrow \Y$ where $\Y \hookrightarrow \W \in \weqv$.  Thus, by applying \cref{lem:equiv_with_left_fracs}, we obtain the desired bi-equivalence.
	\end{proof}
	
	We therefore immediately deduce a criterion for `Morita equivalence' of étale complete topological groupoids:
	
	\begin{coro}\label{morita-eqv-for-ectgrpd}
		Given étale complete topological groupoids $\X, \Y$, there is an equivalence of topoi $\Sh(\X) \simeq \Sh(\Y)$, i.e.\ $\X$ and $\Y$ are \emph{Morita equivalent}, if and only if there exists an étale complete topological groupoid $\W$ and embeddings of subgroupoids $\Phi \colon \X \hookrightarrow \W$ and $\Psi \colon \Y \hookrightarrow \W$ with $\Phi, \Psi \in \weqv$.
	\end{coro}
	
	\begin{rem}
		Therefore, a property of topological groupoids is preserved by the weak equivalences identified in \cref{sec:weqv} precisely if it is a `Morita invariant' property of étale complete topological groupoids, in that if $\X$ satisfies property $P$, then so too does every $\Y$ that is Morita equivalent to $\X$.
	\end{rem}
	
	\begin{rem}\label{rem:advert-for-AZ-extension}
		\cref{morita-eqv-for-ectgrpd} can be applied to so-called `reconstruction theorems' in model theory.  Let $\theory_1, \theory_2$ be countably categorical theories, and let $M, N$ be the associated unique countable models.  Recall that, by the Ryll-Nardzewski theorem \cite[Theorem 7.3.1]{hodges}, every type in $\theory_1$ and $\theory_2$ is isolated, and that both $M$ and $N$ are ultrahomogeneous (\cite[\S 10]{hodges}).  Hence, by \cite[Theorem 3.1]{caramellogalois} (see also \cref{ex:elim_para} and \cref{class-repr-grpds-for-theory}), we have that $\B\Aut(M)$ and $\B\Aut(N)$ classify the theories $\theory_1$ and $\theory_2$ respectively.  If $\Aut(M)$ and $\Aut(N)$ are homeomorphic as topological groups, then it immediately follows that $\topos_{\theory_1} \simeq \B \Aut(M) \simeq \B\Aut(N) \simeq \topos_{\theory_2}$, and so $\theory_1$ and $\theory_2$ are \emph{Morita equivalent} -- by which we mean that $\topos_{\theory_1}\simeq \topos_{\theory_2}$.  Conversely, if $\topos_{\theory_1} \simeq \topos_{\theory_2}$, then by \cref{morita-eqv-for-ectgrpd} there is a cospan of weak equivalences $\Aut(M) \hookrightarrow \W \hookleftarrow \Aut(N)$.  If $\W$ is also a topological group, then by \cref{ex:etale-complete-groups}, it follows that $\Aut(M) \cong \Aut(N)$.
		
		This has a strong resemblance to a classical result in model theory \cite[\S 1]{AZ} which asserts that the theories $\theory_1, \theory_2$ are \emph{bi-interpretable} if and only if there is a homeomorphism of topological groups $\Aut(M) \cong \Aut(N)$.  Bi-interpretability of theories coincides with Morita equivalence under mild assumptions (see \cite[\S 5.2]{mceldowney}).  \cref{morita-eqv-for-ectgrpd} therefore paves the way to a groupoidal extension of the classical result of Ahlbrandt-Ziegler \cite{AZ} (which is orthogonal to the recent extension of Ben Yaacov \cite{yaacov}, also employing topological groupoids).  This extension will be fully explained in a later work.
	\end{rem}

	\section*{Acknowledgements}
	
	This work has its origin in my PhD research while at the University of Insubria and completed while employed at Queen Mary University of London.  My thanks go to my PhD supervisor, Olivia Caramello, and my colleagues at Queen Mary, Ivan Toma\v{s}i\'{c} and Behrang Noohi, for their useful comments.

	\renewcommand{\baselinestretch}{0.5}


\end{document}